\documentclass[oneside,english]{amsart}
\usepackage[T1]{fontenc}
\usepackage[latin9]{inputenc}
\usepackage{color}
\usepackage{babel}
\usepackage{float}
\usepackage{units}
\usepackage{amsthm}
\usepackage{amssymb}
\usepackage{graphicx}
\usepackage{xargs}[2008/03/08]
\usepackage[unicode=true,pdfusetitle,
 bookmarks=true,bookmarksnumbered=false,bookmarksopen=false,
 breaklinks=false,pdfborder={0 0 1},backref=false,colorlinks=true]
 {hyperref}
\usepackage{breakurl}

\makeatletter

\newcommand{\noun}[1]{\textsc{#1}}

\numberwithin{equation}{section}
\numberwithin{figure}{section}
 \newcommand\thmsname{\protect\theoremname}
 \newcommand\nm@thmtype{theorem}
 \theoremstyle{plain}
 
 \newenvironment{namedthm}[1][Undefined Theorem Name]{
   \ifx{#1}{Undefined Theorem Name}\renewcommand\nm@thmtype{theorem*}
   \else\renewcommand\thmsname{#1}\renewcommand\nm@thmtype{namedtheorem}
   \fi
   \begin{\nm@thmtype}}
   {\end{\nm@thmtype}}
  \theoremstyle{remark}
  \newtheorem*{rem*}{\protect\remarkname}
\theoremstyle{plain}
\newtheorem{thm}{\protect\theoremname}[section]
  \theoremstyle{definition}
  \newtheorem{defn}[thm]{\protect\definitionname}
  \theoremstyle{plain}
  \newtheorem{prop}[thm]{\protect\propositionname}
  \theoremstyle{remark}
  \newtheorem{rem}[thm]{\protect\remarkname}
  \theoremstyle{plain}
  \newtheorem{lem}[thm]{\protect\lemmaname}
  \theoremstyle{plain}
  \newtheorem{cor}[thm]{\protect\corollaryname}

\usepackage{kpfonts}
\usepackage{esint}

\@ifundefined{showcaptionsetup}{}{%
 \PassOptionsToPackage{caption=false}{subfig}}
\usepackage{subfig}
\AtBeginDocument{
  
}

\makeatother

  \providecommand{\corollaryname}{Corollary}
  \providecommand{\definitionname}{Definition}
  \providecommand{\lemmaname}{Lemma}
  \providecommand{\propositionname}{Proposition}
  \providecommand{\remarkname}{Remark}
  \providecommand{\theoremname}{Theorem}
\providecommand{\theoremname}{Theorem}

\begin{document}
\global\long\def\tx#1{\mathrm{#1}}
\global\long\def\ww#1{\mathbb{#1}}
\global\long\def\dd#1{\tx d#1}
\newcommandx\fol[1][usedefault, addprefix=\global, 1=R]{\mathcal{F}_{#1}}
\global\long\def\ii{\tx i}
\global\long\def\adh#1{\tx{adh}\left(#1\right)}
\global\long\def\re#1{\Re\left(#1\right)}
\global\long\def\im#1{\Im\left(#1\right)}
\global\long\def\nf#1#2{\nicefrac{#1}{#2}}
\global\long\def\germ#1{\ww C\left\{  #1\right\}  }
\newcommandx\grp[2][usedefault, addprefix=\global, 1=R, 2=\Sigma]{\Gamma_{#1}\left(#2\right)}
\global\long\def\pp#1{\frac{\partial}{\partial#1}}
\newcommandx\sat[2][usedefault, addprefix=\global, 1=\fol]{\tx{Sat}_{#1}\left(#2\right)}
\global\long\def\tt#1{\mathtt{#1}}
\newcommandx\spt[1][usedefault, addprefix=\global, 1=\gamma]{\mathtt{#1}_{\star}}
\global\long\def\hol#1{\frak{h}_{#1}}
\newcommandx\ept[1][usedefault, addprefix=\global, 1=\gamma]{\mathtt{#1}^{\star}}
\global\long\def\id{\tx{Id}}
\global\long\def\norm#1{\left|\left|#1\right|\right|}
\global\long\def\condX{\tt X}
\global\long\def\condA{\tt a}
\global\long\def\condR{\tt R}
\global\long\def\rsupp{\tt{Res}\left(a,\lambda\right)}

\title{Extending to the complex line Dulac's corner maps of non-degenerate
planar singularities }

\date{December 2014}

\author{Loïc TEYSSIER}

\address{Laboratoire I.R.M.A., 7 rue R. Descartes, Université de Strasbourg,
67084 Strasbourg cedex, France}



\thanks{This work was partially supported by the grant ANR-13-JS01-0002-01
of the French National Research Agency.}

\email{\texttt{\href{mailto:teyssier@math.unistra.fr}{teyssier@math.unistra.fr}}}

\urladdr{\href{http://www-irma.u-strasbg.fr/~teyssier/}{http://www-irma.u-strasbg.fr/$\sim$teyssier/}}
\begin{abstract}
We study the complex Dulac map for a holomorphic foliation of the
complex plane, near a non-degenerate singularity (both eigenvalues
of the linearization are nonzero) with two separatrices. Following
the well-known results of \noun{Y.~Il'yashenko} we provide a geometric
approach allowing to study the whole maximal domain of (geometric)
definition of the Dulac map. In particular its topology and the regularity
of its boundary are completely described. We also study the order
of magnitude of the first non-trivial term of its asymptotic expansion
and show how to compute it using path integrals supported in the leaves
of the linearized foliation. Explicit bounds on the remainder are
given. We perform similarly the study of the Dulac time spent around
the singularity. All results are formulated in a unified framework
taking no heed to the usual dynamical discrimination (\emph{i.e.}
no matter whether the singularity is formally orbitally linearizable
or not and regardless of the arithmetic of the eigenvalues ratio).
\end{abstract}

\maketitle

\section{Introduction}

We consider a germ of a holomorphic vector field at the origin of
the complex plane $A\left(x,y\right)\pp x+B\left(x,y\right)\pp y$
admitting an isolated, non-degenerate singularity at $\left(0,0\right)$.
In other words the origin is the only local zero of the vector field,
and its linear part $\left[\nabla A\,,\,\nabla B\right]$ at this
point is a $2\times2$ matrix with two nonzero eigenvalues, of ratio
$\lambda\in\ww C_{\neq0}$. Up to choose differently the local analytic
coordinates (we particularly refer to Lemma~\ref{lem:preparation})
we may assume without loss of generality that the vector field admits
the following expression:
\begin{align*}
X_{R} & =\lambda x\pp x+\left(1+R\right)y\pp y\,\,~,~R\left(0,0\right)=0\,\,,~~R\in x^{a}\germ{x,y}~\tag{\ensuremath{\condX}}
\end{align*}
for some non-negative integer $a$. Our study is carried out on a
fixed polydisc $\mathcal{U}=\rho\ww D\times r\ww D$ small enough
for the relation 
\begin{align*}
\sup_{\mathcal{U}}\left|R\right| & <1\tag{\ensuremath{\condR}}
\end{align*}
to hold. At some point we also use the hypothesis 
\begin{align*}
\re{a+\frac{1}{\lambda}} & \geq0\,.\tag{\ensuremath{\condA}}
\end{align*}
 This setting encompasses almost all non-degenerate singularities,
including every kind of saddle singularities ($\lambda<0$).

\bigskip{}

We write $\fol$ the (holomorphic, singular) foliation of $\mathcal{U}$
whose leaves are defined by the integral curves of $X_{R}$. This
foliation admits two special leaves (called separatrices) each of
whose adherence corresponds to a branch of $\left\{ xy=0\right\} $.
We define 
\begin{eqnarray}
\hat{\mathcal{U}} & := & \mathcal{U}\backslash\left\{ xy=0\right\} ~.\label{eq:pointed_neighborhood}
\end{eqnarray}
 Outside $\left\{ x=0\right\} $ the foliation is transverse everywhere
to the fibers of the fibration 
\begin{eqnarray*}
\Pi\,:\,\left(x,y\right) & \longmapsto & x
\end{eqnarray*}
 and if $\left(\tt R\right)$ holds the foliation is transverse to
that of $\left(x,y\right)\mapsto y$ too. Being given $\left(x_{*},y_{*}\right)\in\hat{\mathcal{U}}$
it is thus possible (under suitable assumptions that will be detailed
later on) to lift in the foliation through $\Pi$ a path $\gamma$
linking $x$ to $x_{*}$, starting from the point $\left(x,y_{*}\right)$.
The arrival end-point of the lifted path defines uniquely a point
$\left(x_{*},y_{x}\right)\in\Pi^{-1}\left(x_{*}\right)$. This construction
yields a locally analytic map from the transverse disc $\left\{ y=y_{*}\right\} $
into the transverse disc $\left\{ x=x_{*}\right\} $, which is known
as the \textbf{Dulac map} 
\begin{eqnarray*}
\mathcal{D}_{R}\,:\,x\neq0 & \longmapsto & y_{x}
\end{eqnarray*}
 of $X_{R}$ associated to $\left(x_{*},y_{*}\right)$. This map is
in general multivalued, and its monodromy is generated by the holonomy
of $\fol[R]$ computed on $\left\{ x=x_{*}\right\} $ by winding around
$\left\{ x=0\right\} $.

\begin{figure}[H]
\includegraphics[width=6cm]{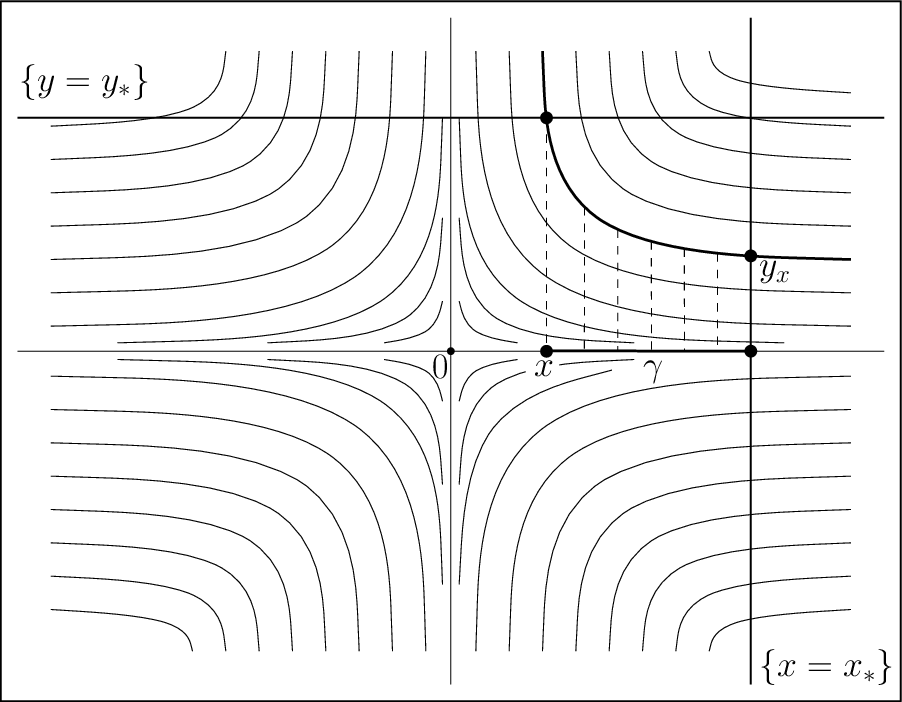}

\caption{The Dulac map, pictured in a real slice.}

\end{figure}

\noun{Y.~Il'Yashenko} carried out important works~\cite{IlYaRus,Ilya,IlyaDu}
aimed at studying the germ of a subdomain of $\left\{ y=y_{*}\right\} $
at $\left(x_{*},y_{*}\right)$ on which the germ of a mapping $\mathcal{D}_{R}$
is holomorphic and has «interesting» asymptotics, in connection to
Dulac's conjecture and ultimately Hilbert's $16^{\tx{th}}$ problem.
We give more contextual details in paragraphs to come, let us just
say for now that \noun{Y.~Il'Yashenko} proved $\mathcal{D}_{R}$
is defined at least on a \emph{standard quadratic domain} in logarithmic
coordinates. By contrast, our concern here is to consider $\mathcal{D}_{R}$
as a global mapping and to describe in a detailed manner the maximal
Riemann surface $\Omega_{R}$ on which $\mathcal{D}_{R}$ is defined
(for which the geometric construction above can be performed\footnote{It may happen that $\mathcal{D}_{R}$ admits a bigger Riemann surface
as a multivalued map, but this case is not dealt with here.}). Our main result is the following:
\begin{namedthm}[Main Theorem]
\label{thm:dulac_topology}Assume conditions $\left(\condX\right)$
and $\left(\condR\right)$ hold.
\begin{enumerate}
\item $\Omega_{R}$ is simply connected and embeds into the universal covering
of $\left\{ y=y_{*}~,~x\neq0\right\} $.
\item Any component of the boundary $\partial\Omega_{R}$ is a piecewise-analytic
curve.
\item If $\re{\lambda}\geq0$ then $\Omega_{R}$ is connected.
\end{enumerate}
\end{namedthm}
\begin{rem*}
~
\begin{enumerate}
\item Notice that the result holds as soon as the quantitative and explicit
conditions $\left(\tt X\right)$,~$\left(\tt R\right)$ are fulfilled:
the theorem is a semi-local result. 
\item If the condition $\left(\condR\right)$ is not met then the foliation
$\fol$ is likely to fail being transverse to the fibration $\Pi$
at some point. The Dulac map becomes multivalued in the presence of
such tangency points, meaning that $\Omega_{R}$ can no longer be
embedded in the universal cover of the punctured transverse line.
The topology of $\Omega_{R}$ can therefore get richer, \emph{e.g.}
when the tangency point corresponds to a finite branch-point: in that
case $\adh{\Omega_{R}}$ contains an orbifold point and $\Omega_{R}$
is not simply connected anymore.
\item Conclusions~(1) and~(3) are still valid when $R$ is merely $C^{1}$
as a real function since their proofs actually only use elementary
variational and topological arguments based on the sole knowledge
of $\sup_{\mathcal{U}}\left|R\right|$, as well as the local rectification
theorem for vector fields. Conclusion~(3) may hold with $\partial\Omega_{R}$
of the same regularity as $R$ but the local finiteness requirement
in <<piecewise>> might be violated in some instances (see in particular
Section~\ref{sub:boundary}).
\end{enumerate}
\end{rem*}

\subsection{Context and known facts }

The Dulac map governs part of the dynamics of $\fol$ and has been
submitted to an intense study at least in the two settings we describe
now.

\subsubsection{The proof of Dulac's conjecture regarding finiteness of the number
of limit cycles for analytic vector fields in the real plane}

Dulac maps are basic ingredients of the cross-section first-return
map along a poly-cycle, whose attractive fixed-points correspond to
limit cycles. As was noticed by \noun{Y.Il'yashenko}~\cite{Ilya}
the original <<proof>> of \noun{H.~Dulac}~\cite{DuDu} crucially
depends on a lemma which turned out to be false. Many powerful, if
intricate, tools have been developed in the 1980's decade which finally
led to a complete proof of Dulac's conjecture. Two parallel approaches
evolved at the time to analyze the asymptotic expansion of the Dulac
map: \noun{J.~Écalle~\cite{EcalDu}} studied it first formally using
trans-series then through resurgent summation techniques, while \noun{Y.~Il'yashenko}
devised an argument based on super-accurate asymptotic series in the
book~\cite{IlyaDu} dedicated to the proof of Dulac's conjecture.
The key point of the argument is that a map with null asymptotic series
should vanishes identically. This is not true for every possible domains,
which is the reason why Dulac's attempt failed. A central class of
domains for which that property does hold are called standard (quadratic)
domains. 

Although various authors contributed to the tale of Dulac's conjecture,
the aim of this article is not to offer a comprehensive list. We make
the choice to refer the reader to the textbook~\cite{IlYakov} for
more details regarding the context in which the Dulac map intervenes,
for this book has a more geometric flavor. We also mention the unpublished
material~\cite{Lolo} containing many dynamical details in the complex
setting.
\begin{namedthm}[Regularity Theorem \cite{IlyaDu,IlYaRus}]
Assume $\lambda<0$. Up to reduce the value of $\rho$ and $r$,
the set $\Omega_{R}$ contains a \textbf{standard quadratic domain
}for $C>0$ big enough: 
\begin{eqnarray}
\Omega\left(C\right) & := & \varphi_{C}\left(\left\{ z~:~\re z>0\right\} \right)\label{eq:standard_domain}
\end{eqnarray}
image of the real half-plane $\left\{ \re z>0\right\} $ by the conformal
mapping 
\begin{eqnarray*}
\varphi_{C}~:~z & \longmapsto & z_{*}-z-C\sqrt{1+z}~.
\end{eqnarray*}
\end{namedthm}
\begin{rem*}
~
\begin{enumerate}
\item The constant $C$ can be made as small as wished by taking $x_{*}$
close enough to $0$. Viewed in the original $x$-variable, this domain
contains germs of a sector around $\left\{ x=0\right\} $ of arbitrary
aperture.
\item The most recent (and so far shortest) version of the proof given in~\cite[chap. IV]{IlYakov}
uses the condition~$\left(\tt R\right)$ (or more precisely the bound
$\left|R\right|<\nf 12$ with $a=0$). The approach ultimately relies
on the fact that the holonomy of $\fol[R]$ is a parabolic germ, which
happens only when $\lambda$ is rational, and no useful replacement
estimate of the behavior of the iterates of the holonomy currently
exist when $\lambda$ is an irrational.
\item In~\cite{IlyaDu,IlYaRus} a proof of the statement for every $\lambda\in\ww R_{<0}$
is performed under the (non-restrictive, see Lemma~\ref{lem:preparation})
assumption that $a>0$ is big enough.
\end{enumerate}
\end{rem*}

\subsubsection{The topology of singular germs of a planar holomorphic foliation}

The study of the dynamics of a singular foliation through Seidenberg's
reduction\footnote{According to Seidenberg's algorithm (see~\cite{Seiden}) any isolated
singularity of a germ of a holomorphic foliation $\mathcal{F}$ can
be <<reduced>> through a proper, rational map $\pi\,:\,\mathcal{M}\to\left(\ww C^{2},0\right)$,
where $\mathcal{M}$ is a conformal neighborhood of a tree $E:=\pi^{-1}\left(0,0\right)$
of normally-crossing, conformal divisors $\ww P_{1}\left(\ww C\right)$.
The pulled-back foliation $\pi^{*}\mathcal{F}$ has only isolated,
reduced singularities (located on $E$): either non-degenerate or
of saddle-node type (exactly one nonzero eigenvalue).} process naturally involves Dulac maps as <<corner maps>> encoding
the transition between different components of the exceptional divisor.
They measure how the different components of the projective holonomy
pseudo-group mix together. In that context \noun{D.~Mar\'{i}n} and
\noun{J.-F.~Mattei~\cite{MarMat} }proved that under suitable (generic)
hypothesis a germ of a singular foliation is locally incompressible:
there exists an adapted base of neighborhoods of $\left(0,0\right)$
in which the (non-trivial) cycles lying in the leaves of the restricted
foliation must wind around the complement of the separatrix locus,
in the trail of Milnor's theorem regarding holomorphic fibrations
outside the singular fibers. This study is the first step towards
a complete analytical and topological classification of (generic)
singular germs of a foliation~\cite{MarMatMono}. 

One of the main ingredients of their proof is the control of the <<roughness>>
of the corner maps and elements of the projective holonomy. This roughness
can be read in the first two terms of the asymptotic expansion of
the Dulac map. In that respect Proposition~\ref{prop:model_dominant},
stated further down, allows to get rid of a non-necessary technical
hypothesis in Mar\'{i}n-Mattei's theorem, namely discarding <<bad\footnote{A «bad» irrational number is unusually well approximated by rational
numbers (the so-called «small-divisors» problem), characterized by
Brjuno's explicit arithmetic condition~\cite{Brju} expressed in
terms of the convergents of $\lambda$.}>> irrational ratios appearing in Seidenberg's reduction of the singularity.
We refer to \cite{MarTey} for a complete dealing with the more general
setting, as well as the proof of:
\begin{namedthm}[Incompressibility Theorem \cite{MarTey}]
Take $\lambda\neq0$ and assume that conditions~$\left(\condX\right)$
and~$\left(\condR\right)$ are fulfilled. Let $\mathcal{E}\,:\,\left(z,w\right)\longmapsto\left(\exp z,\exp w\right)$
be the universal covering of $\hat{\mathcal{U}}$. Then the foliation
$\mathcal{E}^{*}\fol[R]$ is regular and each one of its leaves is
simply-connected.
\end{namedthm}

\subsection{Discussion and additional results}

Let us begin with formulating a few remarks.
\begin{itemize}
\item We provide a framework which does not depend on the usual dynamical
discrimination (\emph{i.e.} no matter whether the singularity is formally
orbitally linearizable or not and regardless of the arithmetic of
$\lambda$). We only distinguish between \textbf{node-like} ($\re{\lambda}\geq0$)
and \textbf{saddle-like singularities} ($\re{\lambda}<0$). Notice
that if $\lambda\notin\ww R$ the singularity is hyperbolic, thus
orbitally linearizable by Poincaré's theorem: the genuinely difficult
cases arise only for real $\lambda$ (actually $\lambda<0$), which
are the cases usually studied in the literature. Yet our constructions
are real-analytic with respect to $\lambda$, allowing irregular situations
to be studied as limiting cases of \emph{e.g.} hyperbolic situations.
\item In the setting of real planar vector fields a \textbf{quasi-resonant
saddle} ($\lambda\in\ww R_{<0}\backslash\ww Q$) is always orbitally
conjugate to its linear part by a (sufficiently more regular than)
$C^{\infty}$ change of coordinates~\cite{Ilya,IlyaDu}, which explains
why part of the literature (\emph{e.g.}~\cite{IlYakov}) mainly focus
on \textbf{resonant saddles} ($\lambda\in\ww Q_{<0}$). This is no
longer true in the complex plane in the presence of <<bad>> irrational
ratios~\cite{Ricard}. Although a quasi-resonant saddle is orbitally
conjugate to its linearization in the universal covering of $\hat{\mathcal{U}}$,
this conjugacy cannot extend to a $C^{\infty}$ map along the lift
of the separatrices. Therefore it is not sufficient to study linear
foliations to encompass complex Dulac maps or their asymptotics in
the case of quasi-resonant saddles.
\item We express below the Dulac map as an integral (more precisely, as
the integral of the differential $1$-form $R\dd{\left(\log x^{-\nf 1{\lambda}}\right)}$
along paths tangent to $\fol$). This integral form, while not being
strictly speaking an «integral representation», will prove useful
when performing computations. Let us call this integration process
the\textbf{ characteristics} of $R$ along $\fol[R]$.
\end{itemize}
Following the Incompressibility Theorem the natural space to study
the Dulac map is the universal covering of $\hat{\mathcal{U}}$ (see~(\ref{eq:pointed_neighborhood})):
\begin{eqnarray*}
\mathcal{E}~:~\tilde{\mathcal{U}} & \longrightarrow & \hat{\mathcal{U}}\\
\left(z,w\right) & \longmapsto & \left(\exp z~,~\exp w\right)~.
\end{eqnarray*}
The Main Theorem~(1) asserts that $\mathcal{D}_{R}$ can be understood
as a local holomorphic function $z\mapsto w_{z}$, after having fixed
once and for all a preimage $\left(z_{*},w_{*}\right)\in\mathcal{E}^{-1}\left(x_{*},y_{*}\right)$
and set $\mathcal{D}_{R}\left(z_{*}\right):=w_{*}$. Because its Riemann
surface $\Omega_{R}$ is simply connected the map $\mathcal{D}_{R}$
actually is holomorphic on an open set of $\left\{ w=w_{*}\right\} $,
still written $\Omega_{R}$, corresponding to those $z$ giving birth
to a path $\gamma_{R}\left(z\right)$ tangent to $\mathcal{E}^{*}\fol[R]$
with starting endpoint $\left(z,w_{*}\right)$ and landing endpoint
$\left(z_{*},w_{z}\right)$.
\begin{namedthm}[Proximity Theorem]
\label{thm:dulac_asymptotics}Assume that conditions~$\left(\condX\right)$
and~$\left(\condR\right)$ hold. 
\begin{enumerate}
\item For every $z\in\Omega_{R}$ we have
\begin{eqnarray*}
\mathcal{D}_{R}\left(z\right) & = & w_{*}+\frac{z_{*}-z}{\lambda}+\frac{1}{\lambda}\int_{\gamma_{R}\left(z\right)}R\circ\mathcal{E}\,\,\dd z\,.
\end{eqnarray*}

\item If $\re{\lambda}<0$ and the condition~$\left(\condA\right)$ is
satisfied then one has the asymptotic approximation 
\begin{eqnarray*}
\int_{\gamma_{R}\left(z\right)}R\circ\mathcal{E}\,\dd z & = & \int_{\gamma_{0}\left(z\right)}R\circ\mathcal{E}\,\dd z+o\left(\left|z\exp\nf{-z}{\lambda}\right|\right)
\end{eqnarray*}
when $\re z$ tends to $-\infty$ with a bounded imaginary part.
\end{enumerate}
\end{namedthm}
Although the Dulac map's asymptotic expansion on standard domains
can be expressed formally as
\begin{eqnarray*}
\mathcal{D}_{R}\left(x\right) & \simeq & \sum_{n,m}D_{n,m}s_{n,m}\left(x\right)\,,~~~~~D_{n,m}\in\ww C~,
\end{eqnarray*}
where
\begin{eqnarray*}
s_{n,m}\left(x\right) & = & \begin{cases}
\frac{x^{n\lambda+m}-x_{*}^{n\lambda+m}}{n\lambda+m} & \mbox{ if }n\lambda+m\neq0\\
\log\frac{x}{x_{*}} & \mbox{ otherwise}
\end{cases}
\end{eqnarray*}
is the Écalle-Roussarie compensator, the expansion converges if, and
only if, the foliation is analytically normalizable\footnote{Situation which arises not so often in the most interesting (quasi-)resonant
cases.}, as was proved by \noun{A.~Mourtada} and\noun{ R.~Moussu} \cite[Proposition~1]{MouMour}.
We compute the first non-trivial term in the asymptotic expansion
for all saddle-like singularities. We also derive explicit bounds
on the remainder in semi-infinite «horizontal» bands $\left\{ \re z<\tx{cst}~,~\left|\im z\right|<\tx{cst}\right\} $).
It is most probably possible to obtain explicit bounds on sub-standard-domains
following the estimate produced in~\cite{IlyaDu,IlYaRus}. I decided
not to include this computation for the sake of concision. Also these
sharper bounds are not needed to deal with incompressibility of generic
degenerate singularities~\cite{MarMat,MarMatMono,MarTey}.

The computation of the characteristics along the model $\fol[0]$
in~(4) can be carried out explicitly. The exact value of $\int_{\gamma_{0}\left(z\right)}\left(x^{n}y^{m}\right)\circ\mathcal{E}\dd z$
does not offer an insightful interest as such (see Section~\ref{sec:asymptotics}).
We can nonetheless deduce from it the dominant part of $\int_{\gamma_{0}\left(z\right)}R\circ\mathcal{E}\dd z$,
which splits into two components: the \emph{regular part }$\int_{z}^{z_{*}}R\left(\exp u,0\right)\dd u$,
which induces a holomorphic function in $\mathcal{U}$ since $R\left(0,0\right)=0$,
and the \emph{resonant part }obtained by selecting in $R$ only well-chosen
monomials. 
\begin{defn}
\label{def_dominant_support}The \textbf{resonant support} $\rsupp$
associated to $\left(a,\lambda\right)$ is 
\begin{itemize}
\item the empty set if $\lambda\notin\ww R_{<0}$,
\item otherwise the subset of $\ww N^{2}$ defined by 
\begin{eqnarray*}
\rsupp & := & \left\{ \left(n,m\right)\in\ww N^{2}\,:\,m>0\,,\,n\geq a\,,\,\left|n\lambda+m\right|<\frac{1}{2n}\right\} \,.
\end{eqnarray*}

\end{itemize}
For $G\left(x,y\right)=\sum_{n\geq0,\,m\geq0}G_{n,m}x^{n}y^{m}\in\germ{x,y}$
we denote by $G_{0}$ the \textbf{regular part} of $G$
\begin{eqnarray*}
G_{0}\left(x\right) & := & G\left(x,0\right)
\end{eqnarray*}
and by $G_{\tt{Res}}$ its \textbf{resonant part} 
\begin{eqnarray*}
G_{\tt{Res}}\left(x,y\right) & := & \sum_{\left(n,m\right)\in\rsupp}G_{n,m}x^{n}y^{m}\,.
\end{eqnarray*}

\end{defn}
It turns out that this support corresponds indeed to resonant or quasi-resonant
monomials, according to the rationality of $\lambda$ (Lemma~\ref{lem:dominant_support}),
which carry without surprise the major part of the non-regular characteristics.
\begin{prop}
\label{prop:model_dominant}Assume that $\re{\lambda}<0$ and conditions~$\left(\condX\right)$,~$\left(\condR\right)$
and~$\left(\condA\right)$ holds. For any $G\in\germ{x,y}$ we have
\begin{eqnarray*}
\int_{\gamma_{0}\left(z\right)}G_{0}\circ\mathcal{E}\dd z & = & G_{0}\left(0\right)\left(z_{*}-z\right)+O\left(\exp z\right)
\end{eqnarray*}
and if moreover $G\in x^{a}\germ{x,y}$ 
\begin{eqnarray*}
\int_{\gamma_{0}\left(z\right)}G_{\tt{Res}}\circ\mathcal{E}\dd z & = & O\left(\left|z\exp\nf{-z}{\lambda}\right|\right)\\
\int_{\gamma_{0}\left(z\right)}\left(G-G_{{\tt{Res}}}-G_{0}\right)\circ\mathcal{E}\dd z & = & O\left(\left|\exp\nf{-z}{\lambda}\right|\right)
\end{eqnarray*}
(here again $O\left(\bullet\right)$ regards the situation when $\re z$
tends to $-\infty$ while $z$ has bounded imaginary part).
\end{prop}
As far as I know the above result is new when $\lambda<0$ is irrational.
Notice that in this case and if $G_{\tt{Res}}$ is finitely supported
then 
\begin{eqnarray*}
\int_{\gamma_{0}\left(z\right)}G_{\tt{Res}}\circ\mathcal{E}\dd z & = & O\left(\left|\exp\nf{-z}{\lambda}\right|\right)\,.
\end{eqnarray*}

\begin{rem}
If the condition~$\left(\tt a\right)$ is not fulfilled then terms
of order of magnitude $\exp az$ will appear instead of $\exp\nf{-z}{\lambda}$.
This noise can be avoided by normalizing $X_{R}$ further (Lemma~\ref{lem:preparation}).
\end{rem}

\subsection{Time spent near the singularity}

In applications (for instance in the study of real analytic planar
vector fields) it is sometimes important to estimate the Dulac time,
that is the time it takes to drift from $\left(z,w_{*}\right)$ to
$\left(z_{*},\mathcal{D}_{R}\left(z\right)\right)$ in the flow of
the vector field. In the case of $X_{R}$ this time is obviously 
\begin{eqnarray*}
\mathcal{T}_{R,1}\left(z,z_{*}\right):=\frac{z-z_{*}}{\lambda} & = & \int_{\gamma_{R}\left(z\right)}\frac{\dd z}{\lambda}\,.
\end{eqnarray*}
Multiplying $X_{R}$ by a holomorphic unit $U$ does not change the
underlying foliation (\emph{i.e. }the Dulac map), although it does
the Dulac time $\mathcal{T}_{R,U}$. The later is obtained by integrating
a time-form\footnote{A meromorphic $1$-form $\tau$ is a time-form for a vector field
$X$ when $\tau\left(X\right)=1$.}. \noun{P.~Marde\v{s}i\'{c}} and \noun{M.~Saavedra} performed
the complex analytic continuation of the Dulac time near resonant
saddle singularities in~\cite{PavaSaav}. Their approach follows
the techniques of~\cite{IlYakov} while providing slightly different
paths of integration $\gamma_{R}\left(z\right)$. This work was followed
by a paper~\cite{MarMarVi} of \noun{P.~Marde\v{s}i\'{c}}, \noun{D.~Mar\'{i}n
}and \noun{J.~Villadelprat}, studying the asymptotic expansion of
the Dulac time with a controlled bound on the remainder in families
of orbitally linearizable real planar vector fields. The results presented
here complete and sharpen these studies in the more general setting
described in the Proximity~Theorem and Proposition~\ref{prop:model_dominant}.

\bigskip{}

One can choose the time-form as 
\begin{eqnarray*}
\tau\left(x,y\right) & := & \frac{\dd x}{\lambda xU\left(x,y\right)}\,,
\end{eqnarray*}
so that the next result holds.
\begin{thm}
\label{thm:dulac_time}Take $U\in\mathcal{O}\left(\mathcal{U}\right)$
such that $U\left(0,0\right)\neq0$. Assume that $\mathcal{U}$ is
chosen in such a way that, in addition to conditions~$\left(\condX\right)$
and~$\left(\condR\right)$, the holomorphic function $U|_{\mathcal{U}}$
never vanishes. Then the Dulac time is holomorphic on $\Omega_{R}$
and
\begin{eqnarray*}
\mathcal{T}_{R,U}\left(z\right) & = & \int_{\gamma_{R}\left(z\right)}\frac{\dd z}{\lambda U\circ\mathcal{E}}\,.
\end{eqnarray*}
If moreover $\re{\lambda}<0$, the condition~$\left(\condA\right)$
holds and $U-U_{0}\in x^{a}y\germ{x,y}$ then, as $\re z$ tends to
$-\infty$ with bounded imaginary part,
\begin{eqnarray*}
\mathcal{T}_{R,U}\left(z\right)-\int_{\gamma_{0}\left(z\right)}\frac{\dd z}{\lambda U_{0}\circ\mathcal{E}} & = & \int_{\gamma_{0}\left(z\right)}\frac{\dd z}{\lambda U_{{\tt{Res}}}\circ\mathcal{E}}+o\left(\left|z\exp\nf{-z}{\lambda}\right|\right)\,.
\end{eqnarray*}

\end{thm}
The integral subtracted on the left-hand side boils down to $\frac{z_{*}-z}{\lambda U\left(0,0\right)}$
if $U-U\left(0,0\right)$ belongs to $x^{a}y\germ{x,y}$. We mention
that this situation can always be enforced by preparing the vector
field:
\begin{lem}
\label{lem:preparation}Let $Z$ be a germ of a holomorphic vector
field near an isolated, non-degenerate singularity with ratio of eigenvalues
$\lambda\notin\ww R_{\geq0}$. Then there exists $a\in\ww N_{>0}$
satisfying~$\left(\condA\right)$ and a choice of local analytical
coordinates such that $Z=UX_{R}$ for two germs of a function satisfying
$R\in x^{a}y\germ{x,y}$ and $U-U\left(0,0\right)\in x^{a}y\germ{x,y}$
with $U\left(0,0\right)\neq0$.
\end{lem}
This lemma is plainly trivial when $\lambda\notin\ww R_{\leq0}$:
in that setting $Z$ is locally analytically conjugate to its linear
part, corresponding to $R=0$ and $U=\tx{cst}$. When $\lambda<0$
is irrational the vector field is formally linearizable and can be
put in the sought form for any finite order $a\in\ww N$, particularly
one such that $a+\frac{1}{\lambda}>0$. When $\lambda=-\frac{p}{q}$
for $p$ and $q$ positive co-prime integers, resonances may appear
and $Z$ may not be even formally orbitally linearizable. These resonances
correspond to pairs $\left(n,m\right)$ of integers belonging to $\left(q,p\right)\ww N$
(those for which $n\lambda+m=0$) in the Taylor expansion of $R$
and $U-U\left(0,0\right)$, and as such cannot appear for an index
$n$ lesser than $q$ or for $m=0$. It is thus possible to cancel
out formally the $\left(q-1\right)$-jet with respect to $x$ and
the $0$-jet with respect to $y$ of the given functions, meaning
we can take $a:=q$. Then $a+\frac{1}{\lambda}\geq0$. The fact that
this formal transform can always be chosen convergent is well known.

\subsection{Extension of the results to other singularities}

Condition~$\left(\condX\right)$ is satisfied except for some cases
when $\left\{ \lambda,\nf 1{\lambda}\right\} \cap\ww N\neq\emptyset$.
The heuristic is that these singularities may not possess sufficiently
many separatrices: the resonant node ($\lambda\neq0$) and the saddle-node
($\lambda=0$, exactly one nonzero eigenvalue) admit only one in general.
The former case is not very interesting since it corresponds to vector
fields which can be analytically reduced to polynomial vector fields
(Poincaré-Dulac normal forms~\cite{Dulac}) for which explicit computations
are easily carried out. The geometry of the foliation itself is quite
tame and completely understood. Save for some minor and technical
complications, the framework we present can be adapted to encompass
this case, although the trouble is not worth the induced lack of clarity
in the exposition.

The case of the saddle-node is richer. In~\cite{MarTey} we prove
that the Incompressibility~Theorem holds in that case too. When the
saddle-node is not divergent (\emph{i.e. }it admits two separatrices)
it can be brought in the form~$\left(\condX\right)$ and the Dulac
map admits an integral representation as in Proximity~Theorem~(1)
\begin{eqnarray*}
\mathcal{D}_{R}\left(z\right) & = & \mathcal{D}_{\mu x^{k}}\left(z\right)+\int_{\gamma_{R}\left(z\right)}\left(R-\mu x^{k}\right)\circ\mathcal{E}\,\frac{\dd z}{\exp\left(kz\right)}
\end{eqnarray*}
where $\left(k,\mu\right)\in\ww N_{>0}\times\ww C$ is the formal
invariant of the saddle-node and $\mathcal{D}_{\mu x^{k}}$ is the
Dulac map for the normal form which can be explicitly computed:
\begin{eqnarray*}
\mathcal{D}_{\mu x^{k}}\left(z\right) & = & w_{*}+\mu\left(z-z_{*}\right)+\frac{\exp\left(-kz\right)-\exp\left(-kz_{*}\right)}{k}\,.
\end{eqnarray*}
 When the singularity is a divergent saddle-node it is possible to
obtain an integral representation as well as a sectoral asymptotic
behavior. We refer also to~\cite{Lolo} for more details. Notice
again that the results are quantitative and hold whenever condition~$\left(\tt R\right)$
is satisfied.

\bigskip{}

After Seidenberg's reduction of its singularity a (germ of a) nilpotent
foliation possesses singular points either of non-degenerate or of
saddle-node type. As a consequence the work done here and in~\cite{Lolo}
is somehow sufficient to analyze more general Dulac maps, although
the difficulty of the task is huge. Yet there is a special case where
it is not necessary to perform the reduction of singularities to be
able to carry out some computations, which is in fact the most general
formulation of the framework we introduce here, corresponding to vector
fields in the form generalizing~$\left(\condX\right)$
\begin{eqnarray*}
X_{R} & = & X_{0}+RY
\end{eqnarray*}
 where:
\begin{itemize}
\item $X_{0}$ and $Y$ are commuting, generically transverse vector fields,
\item $Y$ admits a holomorphic first-integral $u$ with connected fibers,
\item $R\in u^{a}\germ{x,y}$ for some $a>0$.
\end{itemize}
Being given both a transverse disc $\Sigma$ meeting a common separatrix
of $X_{R}$ and $Y$ at some point $p_{*}$, and a transverse $\Sigma'$
corresponding to a trajectory $\left\{ u=u_{*}\right\} $ of $Y$,
we can define the Dulac map of $X_{R}$ joining $\Sigma$ to $\Sigma'$
by lifting paths through the fibration $\left(x,y\right)\mapsto u\left(x,y\right)$.
Then, with equality as multivalued maps on $\Sigma\backslash\left\{ p_{*}\right\} $,
we have the implicit relation
\begin{eqnarray*}
H_{0}\circ\mathcal{D}_{R} & = & H_{0}\circ\mathcal{D}_{0}\times\exp\int_{\gamma_{R}}R\tau
\end{eqnarray*}
where $\tau$ is some time-form of $X_{R}$ and $H_{0}$ a first-integral\footnote{This first-integral can be multivalued, as is the case in the main
situation studied here where $H_{0}\left(x,y\right)=x^{-\nf 1{\lambda}}y$.} of $X_{0}$. With little additional work the techniques used here
can be applied in that context, in particular regarding the shape
of the domain of $\mathcal{D}_{R}$ and, when applicable, its asymptotics.

\subsection{Structure of the article}

This paper only uses elementary techniques and is consequently self-contained. 
\begin{itemize}
\item Section~\ref{sec:geom_prop} is devoted to proving the Main Theorem.
\item This paper goes on with Section~\ref{sec:asymptotics} where the
explicit computation of characteristics $\int_{\gamma_{0}\left(z\right)}G\circ\mathcal{E}\,\dd z$
are performed for the model $\fol[0]$. Yet the core of the section
is the integral formula~(1) of the Proximity Theorem (Section~\ref{sub:Integral-representation})
and the study of the asymptotic deviation between $\int_{\gamma_{R}}G\circ\mathcal{E}\,\dd z\mbox{ }$
and $\int_{\gamma_{0}\left(z\right)}G\circ\mathcal{E}\,\dd z$. Immediate
consequences of this estimation are~(2) of the Proximity Theorem
and the best part of Theorem~\ref{thm:dulac_time}. 
\item We end this paper with the proof of Proposition~\ref{prop:model_dominant}
in Section~\ref{sub:prop}, completing Theorem~\ref{thm:dulac_time}.
\end{itemize}

\subsection{Notations and conventions}
\begin{itemize}
\item Let $K\subset\ww R^{m}$ be a compact set. A mapping $f~:~K\to\ww R^{n}$
will be said real-analytic if it is the restriction of a real-analytic
mapping on an open neighborhood of $K$. 
\item All the paths $\gamma$ we use throughout the paper are, for the sake
of simplicity, piecewise real-analytic maps from some compact interval
$\ww I$ into $\mathcal{U}$. Its starting point (\emph{resp.} ending
point) is written $\spt$ (\emph{resp. }$\ept$). It will always be
possible, though, to perturb $\gamma$ slightly so that it is real-analytic
everywhere when needed.
\item Take a foliation $\mathcal{F}$ defined on a domain $\mathcal{U}$
and some subset $A\subset\mathcal{U}$. The \textbf{saturation} $\sat[\mathcal{F}]A\subset\mathcal{U}$
is the union of all the leaves of $\mathcal{F}$ intersecting $A$.
\item The \textbf{restriction }of the foliation $\mathcal{F}$ to a sub-domain
$\mathcal{V}\subset\mathcal{U}$ is the foliation on $\mathcal{V}$,
written $\mathcal{F}\cap\mathcal{V}$, whose leaves are the connected
components of the trace on $\mathcal{V}$ of the leaves of $\mathcal{F}$.
\item For the sake of concision we make the convention that an object $\mathcal{X}$
hatted with a \emph{tilde} stands for its pull-back in logarithmic
coordinates $\tilde{\mathcal{X}}:=\mathcal{E}^{*}\mathcal{X}$ 
\begin{eqnarray*}
\mathcal{E}~:~\tilde{\mathcal{U}} & \longrightarrow & \hat{\mathcal{U}}\\
\left(z,w\right) & \longmapsto & \left(\exp z~,~\exp w\right)~.
\end{eqnarray*}

\item If $G\in\mathcal{O}\left(\mathcal{U}\right)\cap x^{a}\germ{x,y}$
is bounded we define its norm as 
\begin{eqnarray*}
\norm G & := & \sup_{\mathcal{U}}\left|\frac{G}{x^{a}}\right|~.
\end{eqnarray*}

\item We recall that the rectifying theorem for regular points $p$ of a
foliation $\fol[~]$ asserts the existence of a local analytic chart
$\psi~:~\left(V,p\right)\to\ww C^{2}$ such that $\psi_{*}\left(\fol[~]\cap V\right)$
is a foliation by lines of constant direction. We call $\left(\psi,V\right)$
a \textbf{rectifying chart}.\end{itemize}
\begin{defn}
\label{def_holonomy}Let $\Sigma\subset\mathcal{U}\backslash\left\{ \left(0,0\right)\right\} $
be a cross-section, transversal everywhere to $\fol$ (for short,
a \textbf{transverse} to $\fol$).
\begin{enumerate}
\item We introduce the groupoid $\grp$ of equivalence classes of paths
$\gamma$ tangent to $\fol$ with $\gamma_{\star}\in\Sigma$, up to
tangential homotopy (that is, homotopy within a given leaf of $\fol$)
with fixed end-points. We call it \textbf{the tangential groupoid
of $\fol$ relative to $\Sigma$}.
\item The tangential groupoid of $\fol$ relative to $\Sigma$ is naturally
endowed with a structure of a foliated complex surface, which can
be understood as the foliated universal covering of the saturation
$\sat{\Sigma}$ of $\Sigma$ by the leaves of $\fol$, that is the
locally biholomorphic, onto mapping
\begin{eqnarray*}
\sigma\,:\,\grp & \longrightarrow & \sat{\Sigma}\\
\gamma & \longmapsto & \ept\,.
\end{eqnarray*}

\item Fix a tangent path $\eta\in\grp[][\Sigma]$ and define the abstract
transversal set 
\begin{eqnarray*}
\Gamma_{\Sigma} & := & \left\{ \gamma\in\grp[][\Sigma]~:~\Pi\circ\gamma=\Pi\circ\eta\right\} ~.
\end{eqnarray*}
We call \textbf{holonomy map} of $\fol[R]$ associated to $\left(\eta,\Sigma\right)$
the locally biholomorphic map
\begin{eqnarray*}
\hol{\eta}~:~\Gamma_{\Sigma} & \longrightarrow & \Pi^{-1}\left(\Pi\left(\ept[\eta]\right)\right)\\
\gamma & \longmapsto & \ept[\gamma]~.
\end{eqnarray*}

\item The \textbf{Dulac map} of $\fol[R]$ associated with $\left(x_{*},y_{*}\right)$
is the holomorphic function defined on 
\begin{eqnarray*}
\Gamma_{*}: & = & \left\{ \gamma\in\grp[][\mathcal{U}\cap\left\{ y=y_{*}\,,\,x\neq0\right\} ]\,:\,\Pi\left(\ept\right)=x_{*}\right\} 
\end{eqnarray*}
by
\begin{eqnarray*}
\mathcal{D}_{R}\,:\,\Gamma_{*} & \longrightarrow & \Pi^{-1}\left(x_{*}\right)\\
\gamma & \longmapsto & \ept[\gamma]\,.
\end{eqnarray*}

\end{enumerate}
\end{defn}

\section{\label{sec:Incompressibility}Incompressibility of the leaves}

We recall that $\mathcal{U}$ is some polydisc centered at $\left(0,0\right)$
on which $R$ is bounded and holomorphic. We define $\hat{\mathcal{U}}:=\mathcal{U}\backslash\left\{ xy=0\right\} $
and $\tilde{\mathcal{U}}$ its universal covering through the exponential
map $\mathcal{E}$
\begin{eqnarray*}
\mathcal{E}~:~\tilde{\mathcal{U}} & \longrightarrow & \hat{\mathcal{U}}\\
\left(z,w\right) & \longmapsto & \left(\exp z~,~\exp w\right)~.
\end{eqnarray*}
The Incompressibility~Theorem asserts the leaves of the foliation
$\tilde{\mathcal{F}}:=\mathcal{E}^{*}\fol$ are simply-connected.
We recall briefly the argument of the proof for two reasons: on the
one hand because we need some basic estimates borrowed from said argument
for the rest of the article, on the other hand because it will make
apparent that the proof remains valid when $R$ is merely $C^{1}$
as a real function.

Write
\begin{eqnarray*}
\mathcal{E}^{*}X_{R} & = & \mathcal{E}^{*}X_{0}+R\circ\mathcal{E}\times\mathcal{E}^{*}\left(y\pp y\right)
\end{eqnarray*}
where
\begin{eqnarray*}
\mathcal{E}^{*}X_{0} & = & \lambda\pp z+\pp w\\
\mathcal{E}^{*}\left(y\pp y\right) & = & \pp w\,.
\end{eqnarray*}
The vector field $\mathcal{E}^{*}X_{R}$ is holomorphic and regular
on the infinite complex rectangle 
\begin{eqnarray*}
\tilde{\mathcal{U}} & := & \left\{ \left(z,w\right)\in\ww C^{2}~:~\re z<\ln\rho\,,\,\re w<\ln r\right\} \,.
\end{eqnarray*}
 It induces a foliation $\tilde{\mathcal{F}}$ transversal to the
fibers of 
\begin{eqnarray*}
\Pi\,:\,\left(z,w\right) & \longmapsto & z\,,
\end{eqnarray*}
 hence the leaf $\tilde{\mathcal{L}}_{p_{0}}$ is everywhere locally
the graph $\left\{ w=f\left(z\right)\right\} $ of some unique germ
of a holomorphic function defined in a neighborhood $V\left(p_{0}\right)$
of $\Pi\left(p_{0}\right)$. Because of this property the boundary
of $\tilde{\mathcal{L}}_{p_{0}}$ is included in the boundary of the
domain of study 
\begin{eqnarray*}
\partial\tilde{\mathcal{U}} & = & \left\{ \re z=\ln\rho\mbox{ or }\im w=\ln r\right\} \,.
\end{eqnarray*}
Incompressibility follows from the existence of (a family of) curves
included in $\tilde{\mathcal{L}}_{p_{0}}$ which project by $\Pi$
on line segments of constant direction. A cycle $\gamma$ within $\tilde{\mathcal{L}}_{p_{0}}$
will therefore be pushed along those curves, as if repelled by the
beam of a searchlight, resulting in a homotopy in $\tilde{\mathcal{L}}_{p_{0}}$
with a path $\hat{\gamma}$ whose projection bounds a region of empty
interior. The foliation obtained by restricting $\tilde{\mathcal{F}}$
to the $3$-space $\Pi^{-1}\left(\hat{\gamma}\left(\ww I\right)\right)$
is a $1$-dimensional foliation transverse to the fibers of $\Pi$.
Therefore the leaves are contractible: $\hat{\gamma}$ (and $\gamma$)
is homotopic in $\tilde{\mathcal{L}}_{p_{0}}$ to a point. 
\begin{defn}
(See Figure~\ref{fig:searchlight-beam}.) For $v\in\left\{ \re z<\ln\rho\right\} $,
$0<\delta<\pi$ and $\vartheta\in\ww S^{1}$ the domain 
\begin{eqnarray*}
S\left(v,\vartheta,\delta\right) & := & \left\{ z\,:\,\re z<\ln\rho\,,\,\left|\arg\left(z-v\right)-\arg\vartheta\right|<\delta\right\} 
\end{eqnarray*}
is called a (\textbf{searchlight})\textbf{ beam} of aperture $2\delta$,
direction $\vartheta$ and vertex $v$. If $v=\Pi\left(p_{0}\right)$
we say it is a \textbf{stability beam} when the real part of the lift
in $\tilde{\mathcal{L}}_{p_{0}}$, starting from $p_{0}$, of an outgoing
ray $t\geq0\mapsto v+t\theta$, with $\left|\arg\nf{\theta}{\vartheta}\right|<\delta$,
is decreasing.\end{defn}
\begin{rem}
This particularly means that the outgoing ray lifts \emph{completely}
in $\tilde{\mathcal{L}}_{p_{0}}$ as long as it does not cross $\left\{ \re z=\ln\rho\right\} $. 
\end{rem}
\begin{figure}[H]
\includegraphics[width=8cm]{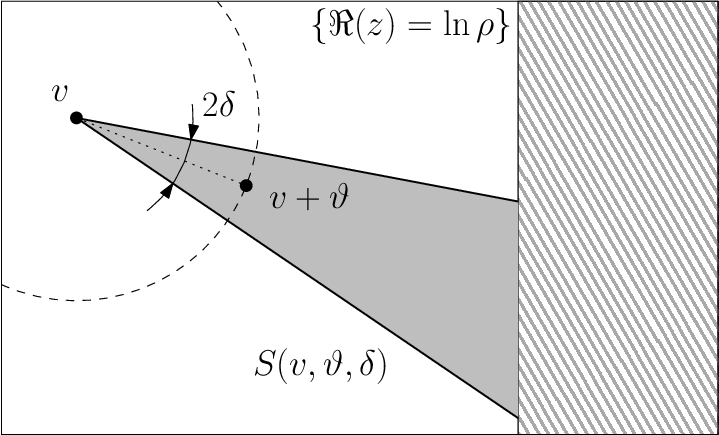}

\caption{\label{fig:searchlight-beam}A searchlight beam.}

\end{figure}

\begin{lem}
\label{lem:_growth_estim_FLF}There exist $\delta\in]0,\pi]$ and
$\vartheta\in\ww S^{1}$ such that for all $p_{0}=\left(z_{0},w_{0}\right)\in\tilde{\mathcal{U}}$
the  beam $S\left(z_{0},\vartheta,\delta\right)$ is a stability beam.
We can take
\begin{eqnarray*}
\vartheta & := & -\frac{\lambda}{\left|\lambda\right|}\\
\delta & \in & \left]0,\arccos\left(\sup_{\mathcal{U}}\left|R\right|\right)\right[\,,
\end{eqnarray*}
so that one can take $\delta$ as close to $\frac{\pi}{2}$ as one
wishes by sufficiently diminishing $\rho$ and $r$. Besides for any
integral curve of $\mathcal{E}^{*}X_{R}$ of the form $t\mapsto\left(z_{0}+t\theta,w\left(t\right)\right)$
with $t\geq0$, $\theta\in\ww S^{1}$ and $w\left(0\right)=w_{0}$
we have the estimate for $\re{\theta}\neq0$ 
\begin{eqnarray}
\left|w\left(t\right)-w_{0}-t\frac{\theta}{\lambda}\right| & \leq & \frac{\exp\left(a\re{z_{0}}\right)}{\left|\lambda\re{\theta}\right|a}\norm R\left|1-\exp\left(at\re{\theta}\right)\right|\label{eq:growth_estim_FLF}
\end{eqnarray}
and taking the limit $\re{\theta}\to0$
\begin{eqnarray}
\left|w\left(t\right)-w_{0}\pm t\frac{\ii}{\lambda}\right| & \leq & \frac{\exp\left(a\re{z_{0}}\right)}{\left|\lambda\right|}t\norm R~.\label{eq:growth_estim_FLF-1}
\end{eqnarray}
\end{lem}
\begin{proof}
The lift in $\tilde{\mathcal{F}}$ of a germ of a ray $z\left(t\right)=z_{0}+\theta t$,
with $\theta\in\ww S^{1}$ and $t\geq0$, starting from $p_{0}$ is
obtained as the solution to
\begin{eqnarray*}
\frac{\dot{w}}{\dot{z}}\left(t\right) & = & \frac{1+R\circ\mathcal{E}\left(z\left(t\right),w\left(t\right)\right)}{\lambda}\,\,\,\,\,,\,w\left(0\right)=w_{0}\,,
\end{eqnarray*}
that is
\begin{eqnarray}
\dot{w}\left(t\right) & = & \frac{\theta}{\lambda}\left(1+R\circ\mathcal{E}\left(z_{0}+\theta t,w\left(t\right)\right)\right)\,.\label{eq:variation_w}
\end{eqnarray}
The function $t\mapsto\varphi\left(t\right):=\re{w\left(t\right)}$
is therefore solution to the differential equation 
\begin{eqnarray}
\dot{\varphi} & = & \re{\frac{\theta}{\lambda}\left(1+R\circ\mathcal{E}\circ\left(z,w\right)\right)}\,,\label{eq:variation_modulus_FLF}
\end{eqnarray}
which particularly means that
\begin{eqnarray*}
\left|\dot{\varphi}\left(t\right)-\re{\frac{\theta}{\lambda}}\right| & \leq & \frac{\exp\left(a\re{z_{0}}\right)}{\left|\lambda\right|}\norm R\exp\left(at\re{\theta}\right)<\frac{\rho^{a}}{\left|\lambda\right|}\norm R\,.
\end{eqnarray*}
Exploiting the cruder estimate by taking $\theta\in\vartheta\exp\left(\ii\left[-\delta,\delta\right]\right)$
we derive
\begin{eqnarray}
\dot{\varphi}\left(t\right) & \leq & -\frac{1}{\left|\lambda\right|}\left(\cos\delta-\rho^{a}\norm R\right)<0\,.\label{eq:derivative_real_w}
\end{eqnarray}
 Since $\varphi\left(0\right)<\ln r$ it follows that $\varphi\left(t\right)<\ln r$
as long as $\re{\left(z\left(t\right)\right)}<\ln\rho$, which is
our first claim.

Integrating both sides of the estimate yields
\begin{eqnarray*}
\left|\re{w\left(t\right)-w_{0}-t\frac{\theta}{\lambda}}\right| & < & \frac{\exp\left(a\re{z_{0}}\right)}{\left|\lambda\re{\theta}\right|a}\norm R\left|1-\exp\left(at\re{\theta}\right)\right|\,.
\end{eqnarray*}
The study we just performed can be carried out in just the same way
for the imaginary part of $w$, yielding
\begin{eqnarray*}
\left|\im{w\left(t\right)-w_{0}-t\frac{\theta}{\lambda}}\right| & < & \frac{\exp\left(a\re{z_{0}}\right)}{\left|\lambda\re{\theta}\right|a}\norm R\left|1-\exp\left(at\re{\theta}\right)\right|\,,
\end{eqnarray*}
proving the sought estimate.\end{proof}
\begin{rem}
\label{rem_roughness}We should stress that the <<roughness>> of
$\partial\Omega_{p_{0}}$ is controlled by the aperture $2\delta$
of the stability beam, which can be taken as close to $\pi$ as one
wishes, and by the direction $\vartheta$ (which is that of the model).
This is a kind of <<conic-convexity>> which forbids $\partial\Omega_{p_{0}}$
to be too wild. In fact the closer $\re z$ is to $-\infty$, the
more $\Omega_{p_{0}}$ looks like $\left\{ \re{w_{0}+\frac{z-z_{0}}{\lambda}}<\ln r\,,\,\re z<\ln\rho\right\} $
near $z$.

\begin{figure}[H]
\includegraphics[width=8cm]{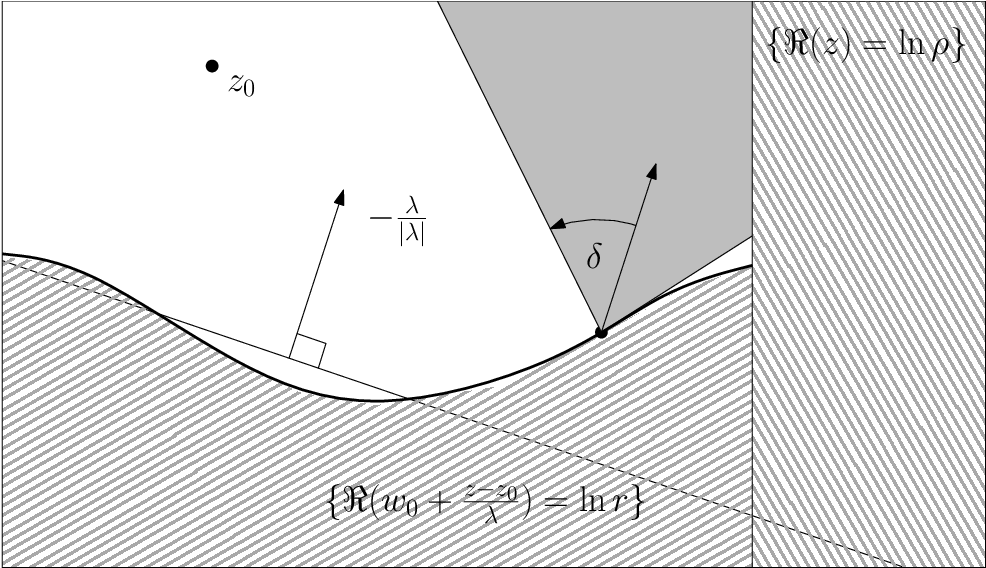}
\end{figure}

\end{rem}
The previous remark can be made more quantitative if we explicitly
allow $\delta$ to vary. In fact we can sharpen the estimate from
the previous lemma by replacing in~(\ref{eq:derivative_real_w})
$\rho^{a}\norm R$ with $M\left(z\left(t\right),w\left(t\right)\right)$
where
\begin{eqnarray*}
M\left(z,w\right) & := & \exp\left(a\re z\right)\left(K_{0}\exp\re z+K_{1}\exp\re w\right)~,\\
K_{0} & := & \sup_{\mathcal{U}}\left|\frac{\partial x^{-a}R}{\partial x}\right|~,\\
K_{1} & := & \sup_{\mathcal{U}}\left|\frac{\partial x^{-a}R}{\partial y}\right|~,
\end{eqnarray*}
for
\begin{eqnarray*}
\left|R\circ\mathcal{E}\left(z,w\right)\right| & \leq & M\left(z,w\right)~.
\end{eqnarray*}
By reducing slightly the size of $\mathcal{U}$ we can assert that
$\sup_{\mathcal{U}}M<1$. The function $M$ depends implicitly on
the parameter $\left(a,K_{0},K_{1}\right)\in\ww N\times\ww R_{\geq0}^{2}$.
\begin{cor}
\label{cor:biggest_domain}Define the functional space 
\begin{eqnarray*}
\ww E & := & \left\{ \epsilon\in C^{0}\left(\ww R_{\geq0}\to\ww R_{>0}\right)~:~\sup_{\ww R_{\geq0}}\epsilon\leq\frac{1-M\left(\ln\rho,\ln r\right)}{\left|\lambda\right|}\mbox{ and }\lim_{t\to\infty}\int_{0}^{t}\epsilon\left(s\right)\dd s=+\infty\right\} ~.
\end{eqnarray*}
For $\left(z,w\right)\in\tilde{\mathcal{U}}$, $\epsilon\in\ww E$
and $t\in\ww R$ set 
\begin{eqnarray*}
\delta_{\epsilon}\left(t,z,w\right) & := & \arccos\left(\epsilon\left(\left|t\right|\right)\left|\lambda\right|+M\left(z,w\right)\right)
\end{eqnarray*}
 (the bound enforced on $\epsilon$ guarantees that $\epsilon\left(\ww R_{\geq0}\right)\left|\lambda\right|+M\left(\tilde{\mathcal{U}}\right)$
is included in $\left[0,1\right]$). Fix $p_{0}=\left(z_{0},w_{0}\right)\in\tilde{\mathcal{U}}$.
As long as the image of the path $z_{\epsilon}^{\pm}~:~t\in\pm\ww R_{\geq0}\mapsto z_{\epsilon}^{\pm}\left(t\right)$,
solution of
\begin{eqnarray*}
\begin{cases}
\frac{\dd{z_{\epsilon}^{\pm}}}{\dd t}\left(t\right) & =\vartheta\exp\left(\pm\ii\delta_{\epsilon}\left(t~,~z_{\epsilon}^{\pm}\left(t\right)~,~w_{0}-\int_{0}^{\left|t\right|}\epsilon\left(s\right)\dd s\right)\right)\\
z_{\epsilon}^{\pm}\left(0\right) & =z_{0}
\end{cases} &  & ~,
\end{eqnarray*}
stays in $\left\{ z~:~\re z<\ln\rho\right\} $ then it is contained
in $\Omega_{p_{0}}$. Moreover if $\re{\lambda}>0$ each $z_{\epsilon}^{\pm}$
is a solution for all $t\in\pm\ww R_{\geq0}$ and
\begin{eqnarray*}
\left|\tan\arg\frac{\dd{z_{\epsilon}^{\pm}}}{\dd t}\left(t\right)\right| & \sim_{t\to\pm\infty} & \frac{\sqrt{1-\left(\epsilon\left(\left|t\right|\right)\left|\lambda\right|\right)^{2}}}{\epsilon\left(\left|t\right|\right)\left|\lambda\right|}~.
\end{eqnarray*}

\end{cor}
The estimate on $\tan\arg\frac{\dd{z_{\epsilon}^{\pm}}}{\dd t}$ as
$t$ goes to $\pm\infty$ controls the asymptotic direction of the
image curve. If $\epsilon\to0$ then the curve gets arbitrarily tangent
to the imaginary axis.
\begin{proof}
Assume $t\geq0$. By construction, and similarly as in~(\ref{eq:variation_modulus_FLF}),
the corresponding solution $w_{\epsilon}^{+}$ satisfies
\begin{eqnarray*}
\frac{\dd{\re{w_{\epsilon}^{+}}}}{\dd t}\left(t\right) & \leq & -\frac{1}{\left|\lambda\right|}\left(\cos\delta_{\epsilon}\left(t,z_{\epsilon}^{+}\left(t\right),w_{0}-\int_{0}^{t}\epsilon\left(s\right)\dd s\right)-M\left(z_{\epsilon}^{+}\left(t\right),w_{0}-\int_{0}^{t}\epsilon\left(s\right)\dd s\right)\right)\\
 &  & =-\epsilon\left(t\right)<0~.
\end{eqnarray*}
 The rest of the proof is clear.\end{proof}
\begin{defn}
We call \textbf{maximal stability beam} of parameter $\left(a,K_{0},K_{1}\right)$
based at $p_{0}$ the open set included in $\Omega_{p_{0}}$
\begin{eqnarray*}
\mathcal{S}_{\max}\left(p_{0}\right) & := & \mbox{connected component of }p_{0}\mbox{ in }\left(\left\{ \re z<\ln\rho\right\} \cap\bigcup_{\epsilon\in\ww E}z_{\epsilon}^{\pm}\left(\pm\ww R_{\geq0}\right)\right)~.
\end{eqnarray*}

\end{defn}
By taking $\frac{R}{x^{a}}$ linear we see that a maximal stability
beam is optimal for all $\fol[R]$ with same corresponding parameter
$\left(a,K_{0},K_{1}\right)$. Notice that $\partial\mathcal{S}_{\max}\left(p_{0}\right)$
is the limit of curves parameterized by $z_{\epsilon}^{\pm}$ with
$\epsilon\to0$ in $\ww E$. In particular its slope gets asymptotically
tangent to the imaginary axis.

\section{\label{sec:geom_prop}Geometry of $\Omega_{R}$ and of its boundary}

We continue to write $\mathcal{D}_{R}$ for the Dulac map of $\fol$
associated to some couple $\left(x_{*},y_{*}\right)$ expressed in
the coordinates $\mathcal{E}$ (\emph{i.e. }understood as a locally
holomorphic function of $\left(z,w\right)$). We recall that for any
$p_{0}\in\tilde{\mathcal{U}}$ the leaf $\tilde{\mathcal{L}}_{p_{0}}$
of $\tilde{\mathcal{F}}$ passing through $p_{0}$ projects on 
\begin{eqnarray*}
\Omega_{p_{0}} & := & \Pi\left(\tilde{\mathcal{L}}_{p_{0}}\right)\,.
\end{eqnarray*}

\begin{prop}
\label{prop:th_2}We fix a preimage $\left(z_{*},w_{*}\right)\in\mathcal{E}^{-1}\left(x_{*},y_{*}\right)$.
\begin{enumerate}
\item The Dulac map is holomorphic on the open, simply connected set 
\begin{eqnarray*}
\Omega & :=\left\{ z\in\ww C\,:\,\left(z,w_{*}\right)\in\tilde{\mathcal{U}}\mbox{ and }z_{*}\in\Omega_{\left(z,w_{*}\right)}\right\} \,.
\end{eqnarray*}
We write $\Omega_{*}$ the connected component containing $z_{*}$.
As Riemann surfaces $\Omega$ and $\Omega_{R}$ are isomorphic.
\item The boundary $\partial\Omega$ is a locally finite union of piecewise
real-analytic curves. 
\item If $\re{\lambda}\geq0$ then $\Omega=\Omega_{*}$ and $\adh{\Omega_{*}}\cap\left\{ \re z=\ln\rho\right\} $
is a nonempty line segment. For every $N\in\ww N_{>0}$ there exists
$r\geq r'>0$ such that this line segment contains at least $\ln\rho+\ii\im{z_{*}}+\left[-\pi\ii N,\pi\ii N\right]$
whenever $\re{w_{*}}<\ln r'$.
\item If $\re{\lambda}<0$ there exists $0<\rho'\leq\rho$ and $0<r'\leq r$
depending only on $a$, $\lambda$ and $\norm R$ such that for every
$\re{z_{*}}<\ln\rho'$, $\re{w_{*}}<\ln r'$ and $N\in\ww N_{>0}$
the domain $\Omega_{*}$ contains some infinite half-band $\left\{ \re z\leq\kappa'\,,\,\left|\im{z-z_{*}}\right|\leq\pi N\right\} $
with $\kappa'\leq\re{z_{*}}$ depending only on $N$, $a$, $\lambda$
and $\norm R$.
\end{enumerate}
\end{prop}
\begin{rem*}
In~(4) one can take $\rho=\rho'$ and $r=r'$ when $\lambda$ is
real. Also when $a>0$ one can guarantee that $r=r'$.
\end{rem*}
The rest of the section is devoted to proving the proposition. In
doing so we build an explicit tangent path linking $\left(z_{0},w_{*}\right)$
to $\left(z_{*},\mathcal{D}_{R}\left(z_{0}\right)\right)$, see Proposition~\ref{prop:integ_path}
below, which will serve in the next section to establish the asymptotic
expansion of the Dulac map through the integral formula of Corollary~\ref{cor:th_1}.
For saddle-like singularities the paths are built in the same fashion
as in \cite{IlYaRus,IlyaDu,IlYakov}. We underline right now the fact
that the projection $\tilde{\gamma}$ of that path through $\Pi$
does not depend on $w_{*}$, but only on $z_{0}$, $a$, $\lambda$,
$\rho$, $\norm R$ and $z_{*}$.

\subsection{The integration path}

We write $p_{0}:=\left(z_{0},w_{*}\right)$. If $\re{\lambda}\geq0$
and $z_{0}\in\Omega$ then both stability beams $S\left(z_{0},\vartheta,\delta\right)$
and $S\left(z_{*},\vartheta,\delta\right)$ are included in $\Omega_{p_{0}}$
and their intersection $W$ is non-empty. Therefore $z_{0}$ can be
linked to $z_{*}$ in $\Omega_{p_{0}}$ by following first a ray segment
of $S\left(z_{0},\vartheta,\delta\right)$ from $z_{0}$ to some point
$z_{1}$ in $W$, then from this point backwards $z_{*}$ along a
ray segment of $S\left(z_{*},\vartheta,\delta\right)$, as illustrated
in Figure~\ref{fig:integration_path-_node} below.

\begin{figure}[H]
\includegraphics[width=8cm]{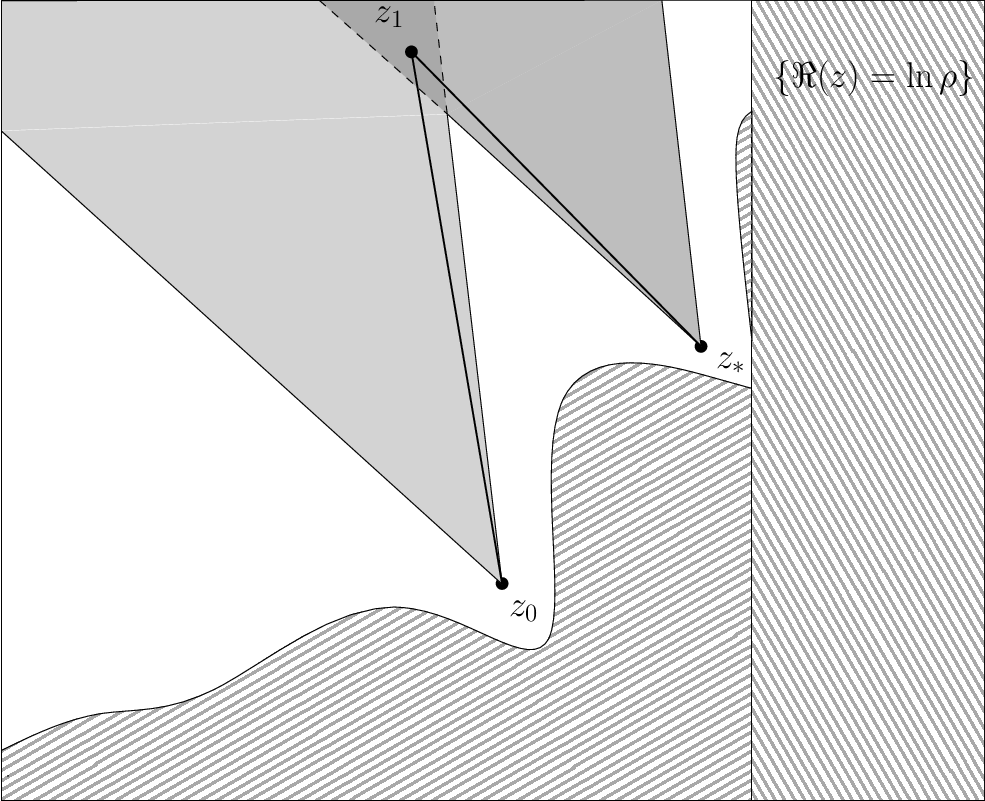}

\caption{\label{fig:integration_path-_node}The path of integration $\tilde{\gamma}$
when $\protect\re{\lambda}\geq0$.}
\end{figure}

On the contrary if $\re{\lambda}<0$ the candidate region $W$ could
be beyond $\left\{ \re z=\ln\rho\right\} $. The construction must
therefore be adapted.
\begin{prop}
\label{prop:integ_path}Assume $\re{\lambda}<0$. There exists $\kappa\in\ww R$
depending only on $a$, $\lambda$, $\norm R$ and $z_{*}$ for which
the following property holds: for every $z_{0}\in\Omega$ one can
choose a path $\tilde{\gamma}\,:\,z_{0}\to z_{*}$ with image inside
$\Omega_{p_{0}}$ in such a way that $\tilde{\gamma}$ is a polygonal
line of ordered vertexes $\left(z_{0},z_{1},z_{2},z_{3},z_{*}\right)$
with (we refer also to Figure~\ref{fig:integration_path} below)
\begin{itemize}
\item $z_{1}=\max\left\{ \kappa,\re{z_{0}}\right\} +\ii\im{z_{0}}$,
\item $\arg\left(z_{2}-z_{1}\right)=\arg\vartheta\pm\delta$,
\item $\re{z_{2}}=\re{z_{3}}<\ln\rho$ and $\left|\im{z_{3}-z_{2}}\right|\leq\left|\im{z_{0}-z_{*}}\right|+\tan\left(\left|\arg\vartheta\right|+\delta\right)\left(\ln\rho-\kappa\right)$,
\item $\arg\left(z_{*}-z_{3}\right)=\arg\vartheta\pm\delta$~.
\end{itemize}
\end{prop}
\begin{figure}[H]
\includegraphics[width=8cm]{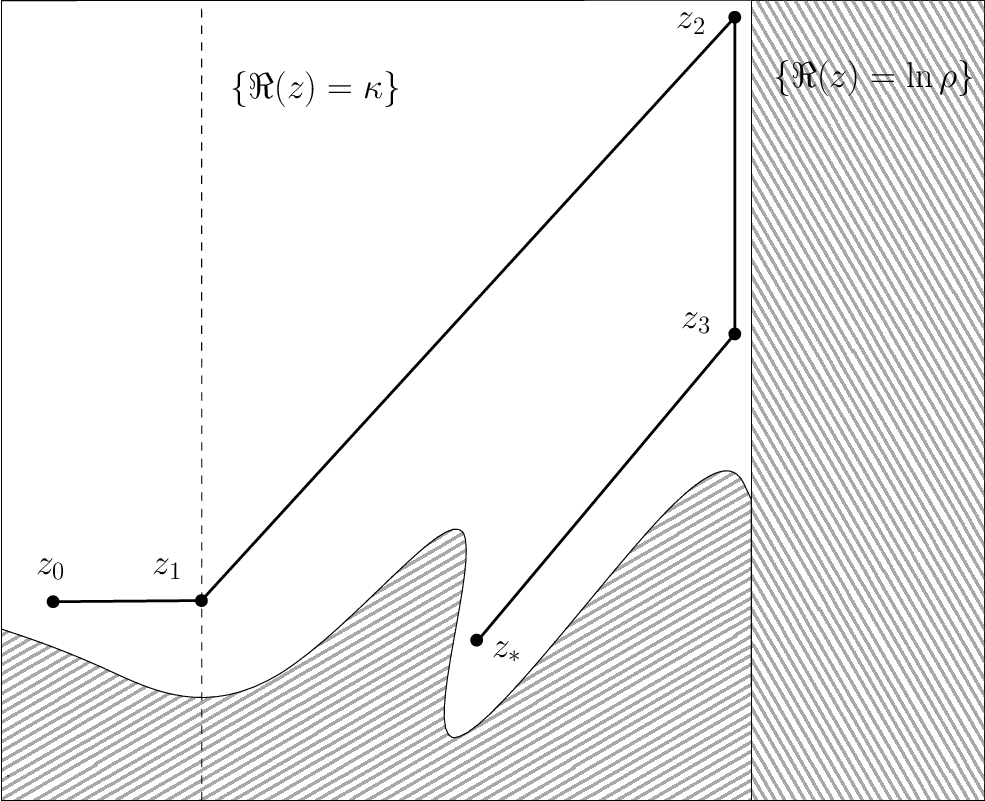}

\caption{\label{fig:integration_path}The path of integration $\tilde{\gamma}$
when $\protect\re{\lambda}<0$.}
\end{figure}

\begin{rem}
We could have made a similar construction without the use of $\kappa$
(\emph{i.e.} by joining directly $z_{0}$ to some $z_{2}$), but we
need it in order to obtain uniform bounds with respect to $\im{z_{0}-z_{*}}$
in the next section.\end{rem}
\begin{proof}
Write 
\begin{eqnarray}
\vartheta_{\pm} & := & \vartheta\exp\left(\pm\ii\delta\right)\,.\label{eq:theta_mp}
\end{eqnarray}
Either $\re{\vartheta_{+}}$ (if $\im{\lambda}\leq0$) or $\re{\vartheta_{-}}$
(if $\im{\lambda}\geq0$) is positive, let us assume for the sake
of example that $\im{\lambda}\leq0$, the other case being similar
in every respect. There exists $\kappa\leq\re{z_{*}}$ such that the
ray segment $\left(z_{0}+\ww R_{\geq0}\right)\cap\left\{ \re z<\kappa\right\} $
is included in $\Omega_{p_{0}}$. Obviously $\kappa$ depends only
on $a,\,\lambda$, $\norm R$ and $z_{*}$. We take for $\tilde{\gamma}$
the polygonal line of ordered vertexes $\left(z_{0},z_{1},z_{2},z_{3},z_{*}\right)$
built in the following fashion.
\begin{itemize}
\item If $z_{0}\in\left\{ \re z<\kappa\right\} $ then the partial ray $\left(z_{0}+\vartheta_{0}\ww R_{\geq0}\right)\cap\left\{ \re z<\kappa\right\} $,
included in $\Omega_{p_{0}}$ according to Lemma~\ref{lem:_growth_estim_FLF},
leaves the region at some point $z_{1}$ with $\re{z_{1}}=\kappa$.
Otherwise we set $z_{1}:=z_{0}$.
\item Both rays $\left\{ z_{1},z_{*}\right\} +\vartheta_{+}\ww R_{\geq0}$,
included in $\Omega_{p_{0}}$, intersect the line $\left\{ \re z=\ln\rho-\epsilon\right\} $
in, respectively, $z_{2}$ and $z_{3}$ for $\epsilon>0$ very small. 
\end{itemize}
The line segment $\left[z_{2},z_{3}\right]$, and therefore the whole
image of $\tilde{\gamma}$, is included in $\Omega_{p_{0}}$ thanks
to the next lemma:
\begin{lem}
\label{lem:boundary_z_FLF}If $\re{\lambda}\leq0$ then $\adh{\Omega_{p_{0}}}\cap\left\{ \re z=\ln\rho\right\} $
is a nonempty line segment. If $\lambda\notin\ww R_{\geq0}$ there
exists $r,~\rho>0$ such that the same property holds.\end{lem}
\begin{proof}
In the case $\re{\lambda}\leq0$ we have $\max\left\{ \re{\vartheta\exp\left(\pm\ii\delta\right)}\right\} >0$;
say, for the sake of example, that $\vartheta_{+}:=\vartheta\exp\left(\ii\delta\right)$
has positive real part. If $\lambda$ is not a positive number this
property can be secured by decreasing $\rho,~r$ and taking $\delta$
as close to $\frac{\pi}{2}$ as need be. Take a path $\Gamma$ connecting
two points of $\adh{\Omega_{p_{0}}}\cap\left\{ \re z=\ln\rho\right\} $
(which is a non-empty set) within $\Omega_{p_{0}}$. Let $I$ be the
line segment of $\left\{ \re z=\ln\rho\right\} $ joining those points.
The ray $p-\vartheta_{+}\ww R_{\geq0}$ emitted from some $p\in I$
separates $\left\{ \re z\leq\ln\rho\right\} $ into two connected
regions. Since $\Gamma$ starts from one of them and lands in the
other one, the curve must cross the ray at some point $q\in\Omega_{p_{0}}$.
The ray $q+\vartheta_{+}\ww R_{\geq0}$ is included in $\Omega_{p_{0}}$
since it lies within a stability beam, while it contains $p$ in its
adherence.
\end{proof}
\end{proof}
We deduce the following characterization.
\begin{cor}
\label{cor:carac_definition_Dulac}Assume $\re{\lambda}<0$. The following
propositions are equivalent.
\begin{enumerate}
\item $z_{0}\in\Omega$,
\item $z_{*}\in\Omega_{p_{0}}$,
\item there exists $\varepsilon>0$ such that for all $0<\epsilon\leq\varepsilon$
the points $z_{2}$ and $z_{3}$ built in Proposition~\ref{prop:integ_path}
can be taken with $\re{z_{2}}=\re{z_{3}}=\ln\rho-\varepsilon$, meaning
$\left[z_{2},z_{3}\right]\subset\Omega_{p_{0}}$.
\end{enumerate}
\end{cor}
\begin{proof}
$\left(1\right)\Leftrightarrow\left(2\right)$ is the definition of
$\mathcal{D}_{R}$ and $\Omega$ while $\left(3\right)\Rightarrow\left(2\right)$
is clear. The converse follows from the previous proposition and its
proof, particularly Lemma~\ref{lem:boundary_z_FLF}.
\end{proof}

\subsection{The dual searchlight sweep}
\begin{lem}
\label{lem:dual_searchlight}If $\re{\lambda}\geq0$ the beam $S\left(z_{0},-\vartheta,\delta\right)$
is included in $\Omega$ for any $z_{0}\in\Omega$. If $\re{\lambda}<0$
 beam $S\left(z_{*},-\vartheta,\delta\right)$ is included in $\Omega_{*}$.
\end{lem}
For any $z\in S\left(z_{0},-\vartheta,\delta\right)$ we can link
$\left(z,w_{*}\right)$ to some point $\left(z_{0},w\right)$ with
$\re w\leq\re{w_{*}}$ by lifting in $\tilde{\mathcal{F}}$ the line
segment $\left[z,z_{0}\right]$. Therefore the lemma is trivial in
the case where $\re{\lambda}<0$. On the contrary when $\re{\lambda}\geq0$
the lemma is a consequence of the next one.
\begin{lem}
Assume that $\re{\lambda}\geq0$, $z_{0}\in\Omega$ and let $\eta:=\re{w_{*}}$.
Then for any other choice of $w_{*}$ with real part lesser or equal
to $\eta$ we still have $z_{0}\in\Omega$ as well.\end{lem}
\begin{proof}
We set up a connectedness argument. Let $B:=\left\{ w_{*}\,:\,\re{w_{*}}\leq\eta\right\} $
and $A:=\left\{ w_{*}\,:\,w_{*}\in B\mbox{ and }z_{0}\in\Omega\right\} $.
By assumption $A$ is not empty, and it is open in $B$ for the same
reason that $\Omega$ is open. More precisely any $w_{*}\in A$ admits
a neighborhood $V$ in $B$ such that the image of $\tilde{\gamma}$
is included in $\Omega_{\left(z_{0},w\right)}$ for every $w\in V$.
Let now a sequence $\left(w_{n}\right)_{n\in\ww N}\subset A$ converge
towards $w_{\infty}\in B$. If $\re{\lambda}\geq0$ then the image
of $\tilde{\gamma}$ is included in the union of the two stability
beams $S:=S\left(z_{0},\vartheta,\delta\right)\cup S\left(z_{*},\vartheta,\delta\right)$
which are themselves included in every $\Omega_{\left(z_{0},w_{n}\right)}$.
Because the real analytic curves defining $\partial\Omega_{\left(z_{0},w\right)}$
vary continuously when $w$ does we have $S\subset\Omega_{\left(z_{0},w_{\infty}\right)}$
also. In particular $z_{*}\in\Omega_{\left(z_{0},w_{\infty}\right)}$
and $z_{0}$ belongs to $\Omega_{*}$ for $w_{*}:=w_{\infty}$. The
former property implies in turn that $A$ is a closed subset of $B$
and as such spans the whole region $B$. \end{proof}
\begin{rem}
When $\re{\lambda}<0$ the above argument does not work since the
image of $\tilde{\gamma}$ must sometimes leave the (adherence of
the) union of stability beams emitted by $z_{1}$ and $z_{*}$ (when
it visits $\left[z_{2},z_{3}\right]$). Nothing guarantees that the
limiting domain $\Omega_{\left(z_{0},w_{\infty}\right)}$ does not
disconnect $\left[z_{2},z_{3}\right]$ at some point. Since this lemma
will be used to show simple-connectedness of $\Omega_{*}$ in the
node-like case, we will need another argument in the saddle-like case.
\end{rem}

\subsection{Proof of Proposition~\ref{prop:th_2}~(1)}

We first mention that $\Omega$ is clearly open since if one can link
a point $\left(z,w_{*}\right)$ to $\left(z_{*},\mathcal{D}_{R}\left(z\right)\right)$
with a tangent path $\gamma$, whose image is included in the open
set $\tilde{\mathcal{U}}$, then surely this is again the case for
a neighborhood of $z$ (consider a covering of the compact image of
$\gamma$ by rectifying charts). Take now a simple loop $\Gamma_{0}\subset\Omega$
bounding a relatively compact, simply-connected domain $W_{0}$ and
show $W_{0}\subset\Omega$ (without loss of generality we may assume
that $\Gamma_{0}$ is smooth and real-analytic). If $\re{\lambda}\geq0$
Lemma~\ref{lem:dual_searchlight} proves precisely that fact, since
\begin{eqnarray*}
W_{0} & \subset & \bigcup_{z\in\Gamma_{0}}S\left(z,-\vartheta,\delta\right)\subset\Omega~.
\end{eqnarray*}
Assume next that $\re{\lambda}<0$ and denote by $\tilde{\gamma}_{z}$
the integration path built in Proposition~\ref{prop:integ_path}
for $z_{0}:=z$, while $\left(z\right)_{j}$ stand for the corresponding
vertex $z_{j}$ of the polygonal line. We want to show that $\tilde{\gamma}_{z}$
can be lifted in $\tilde{\mathcal{F}}$ through $\Pi$ starting from
$\left(z,w_{*}\right)$ all the way to $\left(z_{*},\mathcal{D}_{R}\left(z\right)\right)$,
when $z\in W_{0}$. Along the line segments $\left[z,\left(z\right)_{1}\right]$
and $\left[\left(z\right)_{1},\left(z\right)_{2}\right]$ the real
part of the $w$-component of the lift is decreasing, therefore $\tilde{\gamma}_{z}$
can be lifted at least until $\left(z\right)_{2}$ for every $\re z<\ln\rho$. 

From now on we work in the $3$-dimensional real slice 
\begin{eqnarray*}
\mathcal{R} & := & \left\{ \left(z,w\right)~:~\re z=\ln\rho-\epsilon,~\re w<\ln r\right\} 
\end{eqnarray*}
for fixed $\epsilon>0$, which we identify to
\begin{eqnarray*}
\ww R\times\ww C & = & \left\{ \left(t,w\right)~:~t=\im z,~\left(z,w\right)\in\mathcal{R}\right\} ~.
\end{eqnarray*}
We recall that we write $\tilde{\mathcal{F}}\cap\mathcal{R}$ the
restriction of $\tilde{\mathcal{F}}$ to $\mathcal{R}$, which is
a $1$-dimensional real-analytic regular foliation everywhere transverse
to the fibers of $\Pi|_{\mathcal{R}}$ (the lines $\left\{ t=\tx{cst}\right\} $).
\begin{lem}
\label{lem:surface_W2}We refer to Figure~\ref{fig:omega_sconnected}.
Take a relatively compact, simply connected domain $W\subset\left\{ \re z<\ln\rho-\epsilon\right\} $
with smooth real-analytic boundary $\Gamma$. Consider the correspondence
map 
\begin{eqnarray*}
\hol{0\to2}~:~\adh W & \longrightarrow & \mathcal{R}\\
z & \longmapsto & \hol{\left[z,\left(z\right)_{1},\left(z\right)_{2}\right]}\left(z\right)
\end{eqnarray*}
where $\hol{\gamma}$ denotes the holonomy of $\tilde{\mathcal{F}}$
associated to $\left(\gamma,\left\{ w=w_{*}\right\} \right)$ as in
Definition~\ref{def_holonomy}. 
\begin{enumerate}
\item $\hol{0\to2}$ is a real-analytic, open and injective mapping.
\item Set
\begin{eqnarray*}
\Gamma_{2} & := & \hol{0\to2}\left(\Gamma\right)\subset\mathcal{R}\\
W_{2} & := & \hol{0\to2}\left(W\right)\subset\mathcal{R}~.
\end{eqnarray*}
The compact set $\adh{W_{2}}$ is a smoothly-embedded real-analytic
disc with boundary $\Gamma_{2}$.
\end{enumerate}
\end{lem}
\begin{figure}[H]
\subfloat[~]{\includegraphics[width=5cm]{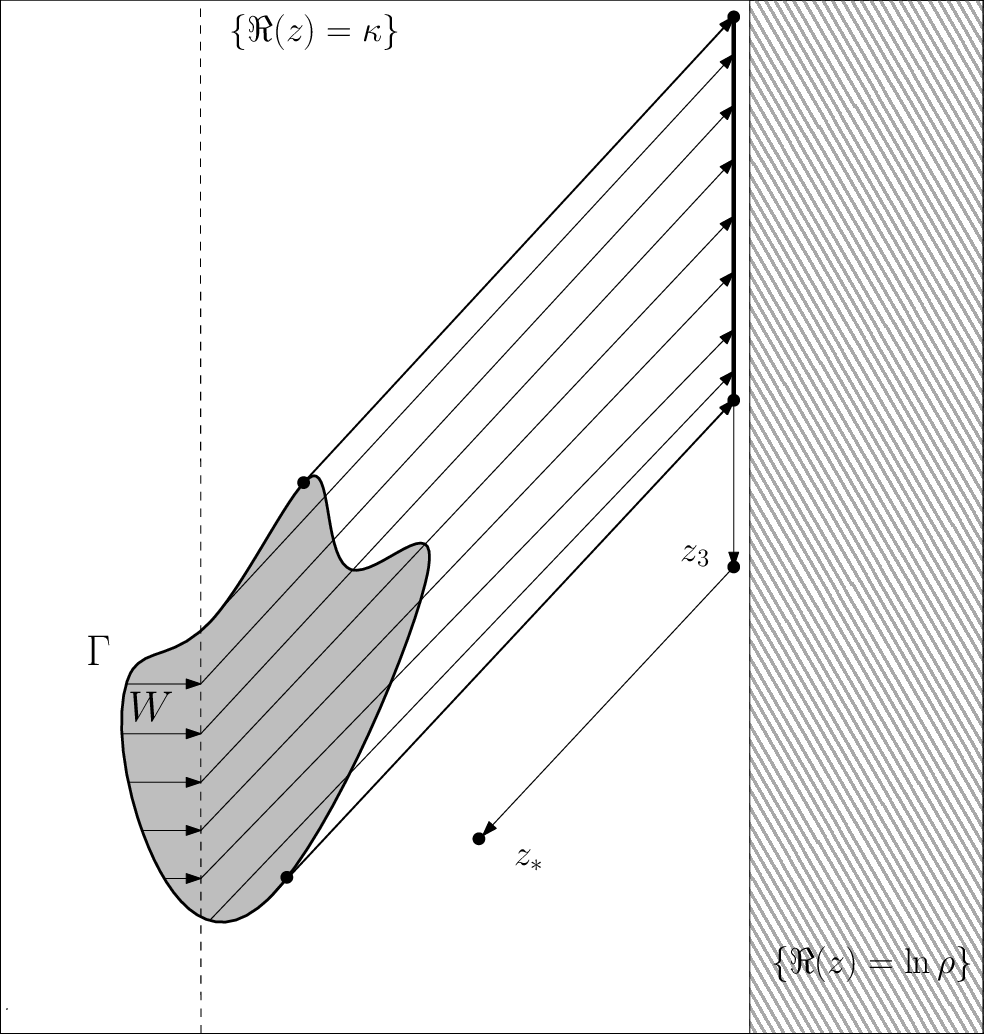}

}\hfill{}\subfloat[~]{\includegraphics[width=7cm]{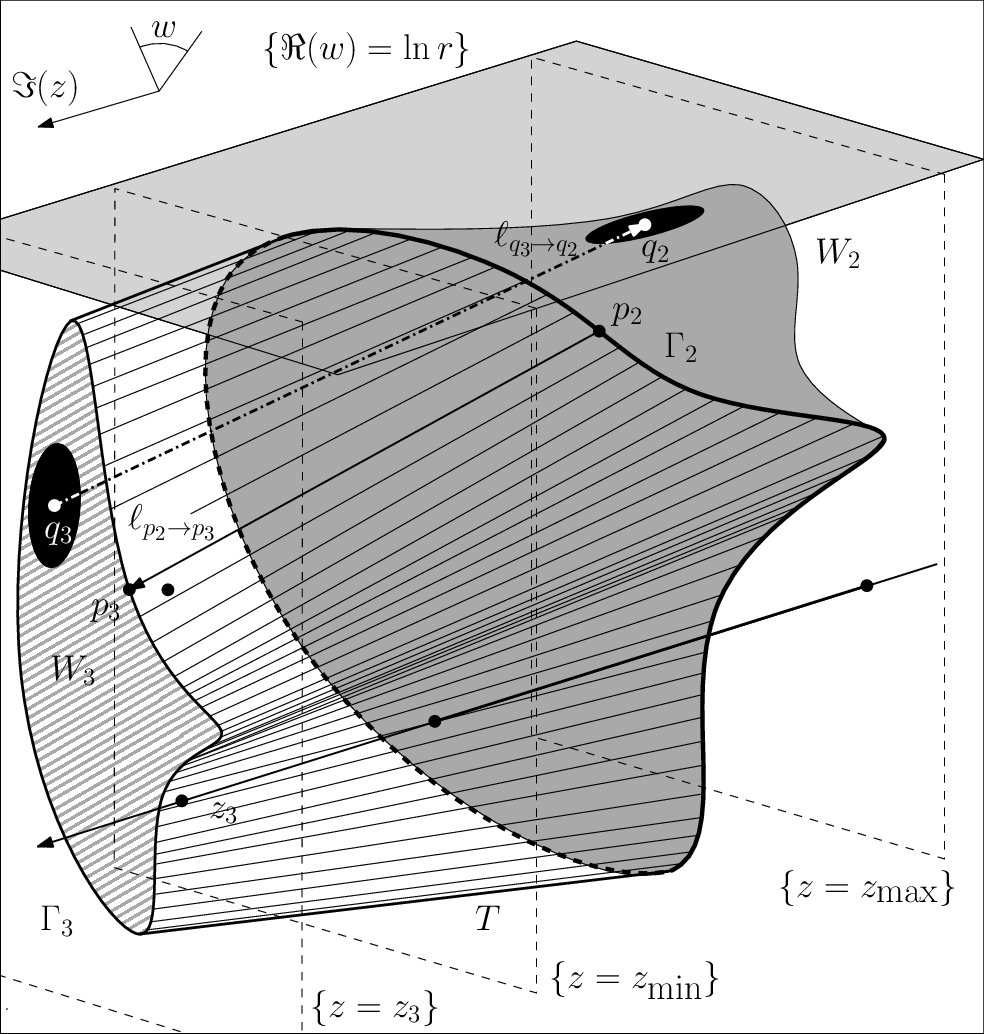}

}

\caption{\label{fig:omega_sconnected}}
\end{figure}

\begin{proof}
Write $\hol{0\to2}\left(z\right)=:\left(\im{\left(z\right)_{2}},w_{2}\left(z\right)\right)$.
\begin{enumerate}
\item The mapping $\hol{0\to2}$ is clearly locally real-analytic and open,
because $w_{2}$ is holomorphic and non-constant. In particular $\partial W_{2}=\Gamma_{2}$.
Show that $\hol{0\to2}$ is injective. If $\adh{W_{2}}$ contains
a double-point $\hol{0\to2}\left(z\right)=\hol{0\to2}\left(z'\right)$
with, say, $\re z\leq\re{z'}$ then $z$ and $z'$ belong to the same
polygonal line $\left[z,\left(z\right)_{1},\left(z\right)_{2}\right]$.
Since the real part of the $w$-component of the lift of the polygonal
line is strictly decreasing we can only have $z=z'$. 
\item The parametric tangent space of $W_{2}$ at $p=\hol{0\to2}\left(z\right)$
is spanned by\linebreak{}
 $\left[\begin{array}{c}
0\\
\frac{\partial w_{2}}{\partial\re z}\left(z\right)
\end{array}\right]$ and $\left[\begin{array}{c}
1\\
\frac{\partial w_{2}}{\partial\im z}\left(z\right)
\end{array}\right]$ as $\frac{\partial\left(z\right)_{2}}{\partial\im z}=\ii$ (we assume
here for the sake of simplicity that $\re z<\kappa$ so that $\frac{\partial\left(z\right)_{2}}{\partial\re z}=0$).
From Cauchy-Riemann formula we know that $\frac{\partial w_{2}}{\partial\re z}\left(z\right)=0$
if, and only if, $w_{2}'\left(z\right)=0$. This outcome is not possible
because $w_{2}$ is locally invertible, hence $\adh{W_{2}}$ is a
smooth real-analytic surface.
\end{enumerate}
\end{proof}
\begin{rem}
This is the only place where we explicitly use the holomorphy of $R$
\emph{via} Cauchy-Riemann formula, although it is not necessary. In
the case where $R$ should only be assumed $C^{1}$ as a real mapping,
the preceding proof can be adapted because $w_{2}$ is a $C^{2}$-diffeomorphism
and therefore $\frac{\partial w_{2}}{\partial\re z}\left(z\right)$
cannot vanish. 
\end{rem}
Although $\tilde{\mathcal{F}}\cap\mathcal{R}$ need  not be transverse
everywhere to $\adh{W_{2}}$, the tangency points are nonetheless
simple as asserted by the following lemma.
\begin{lem}
\label{lem:W2_transverse}For every $p_{2}\in W_{2}$ there exist
a neighborhood $V$ of $p_{2}$ in $\mathcal{R}$ such that any leaf
of $\tilde{\mathcal{F}}\cap V$ intersects $\adh{W_{2}}$ at most
in a single point.\end{lem}
\begin{proof}
Let $p_{2}=\hol{0\to2}\left(z_{0}\right)\in\adh{W_{2}}$ be given
and let $\gamma_{0}$ be the curve linking $\left(z_{0},w_{*}\right)$
to $p_{2}$ along $\tilde{\mathcal{F}}$ above $\left[z_{0},z_{2}\right]$.
We take a finite covering $C=\bigcup_{\iota\leq d}C_{\iota}$ of $\gamma_{0}$
by rectifying holomorphic charts $\left(\psi_{\iota}~:~C_{\iota}\to\ww C^{2}\right)_{\iota\leq d}$
of $\tilde{\fol[~]}$. By relabeling the collection if necessary we
assume that $\left(z_{0},w_{*}\right)\in C_{0}$. We may choose $C_{0}$
small enough so that for any $\left(z,w_{*}\right)\in\left(\adh W\times\left\{ w_{*}\right\} \right)\cap C_{0}$
the tangent curve linking $\left(z,w_{*}\right)$ to $\hol{0\to2}\left(z\right)$
above $\left[z,\left(z\right)_{2}\right]$ is included in $C$. Since
$\tilde{\mathcal{F}}$ is transverse to the fibers of $\left(z,w\right)\mapsto w$
if condition $\left(\tt R\right)$ holds, $\adh W\cap C_{0}$ is transverse
to $\tilde{\fol[~]}\cap C_{0}$: for $C_{0}$ small enough the set
$\adh W\cap C_{0}$ meets any leaf of $\tilde{\fol[~]}\cap C_{0}$
in at most one point. The conclusion follows as $\hol{0\to2}$ is
injective and maps $\adh W$ into $\adh{W_{2}}$. It suffices to take
$V$ such that 
\begin{eqnarray*}
V\cap\adh{\left(W_{2}\right)} & = & \hol{0\to2}\left(\left\{ z~:~\left(z,w_{*}\right)\in C_{0}\right\} \right)~.
\end{eqnarray*}

\end{proof}
We apply Lemma~\ref{lem:surface_W2} to $W:=W_{0}$ (so that $\Gamma=\Gamma_{0}\subset\Omega$)
in order to built $W_{2}\cup\Gamma_{2}$ through the correspondence
map $\hol{0\to2}$ (we may choose $\epsilon$ independent on $z_{0}:=z\in W_{0}\cup\Gamma_{0}$
in Corollary~\ref{cor:carac_definition_Dulac} and use this value
to define $h_{0\to2}$). The key point to complete the proof of Proposition~\ref{prop:th_2}~(1)
is that $\left(z\right)_{3}=z_{3}$ does not actually depend on $z$
once $z_{*}$ is chosen, therefore the process of lifting $\left[\left(z\right)_{2},z_{3}\right]$
in $\tilde{\mathcal{F}}$ takes place solely in $\mathcal{R}$. Because
$\Gamma_{0}\subset\Omega$ the map
\begin{eqnarray*}
\hol{2\to3}~:~\Gamma_{2} & \longrightarrow & \left\{ z=z_{3}\right\} \cap\mathcal{R}\\
p_{2} & \longmapsto & \hol{\left[\Pi\left(p_{2}\right),z_{3}\right]}\left(p_{2}\right)
\end{eqnarray*}
is well-defined. Set
\begin{eqnarray*}
\Gamma_{3} & := & \left\{ \hol{\left[\Pi\left(p\right),z_{3}\right]}\left(p\right)~:~p\in\Gamma_{2}\right\} \subset\left\{ z=z_{3}\right\} \cap\mathcal{R}
\end{eqnarray*}
($\Gamma_{3}$ is a smooth, simple real-analytic loop). We need to
ensure the existence of $W_{3}\subset\left\{ z=z_{3}\right\} \cap\mathcal{R}$
built in the same way from $W_{2}$. In fact we prove the stronger
statement below:
\begin{lem}
\label{lem:surface_W3}The map $\hol{2\to3}$ extends to a bijective
real-analytic map from $\adh{W_{2}}$ onto the compact connected component
$\adh{W_{3}}\subset\left\{ z=z_{3}\right\} \cap\mathcal{R}$ enclosed
by $\Gamma_{3}$. \end{lem}
\begin{proof}
We call $\ell_{p_{2}\to p_{3}}$ the piece of the leaf of $\tilde{\mathcal{F}}\cap\mathcal{R}$
linking $p_{2}\in\Gamma_{2}$ to $p_{3}\in\Gamma_{3}$ above $\left[\Pi\left(p_{2}\right),z_{3}\right]$
(for the sake of clarity the leaves of $\tilde{\mathcal{F}}\cap\mathcal{R}$
are shown as straight lines in Figure~\ref{fig:omega_sconnected}~(B)).
Let $T$ be the smooth cylinder obtained as
\begin{eqnarray*}
T & := & \bigcup_{p_{2}\in\Gamma_{2}}\ell_{p_{2}\to p_{3}}
\end{eqnarray*}
 and $\hat{T}$ be the capped cylinder $\hat{T}:=T\cup\adh{W_{2}}\cup\adh{W_{3}}$
which is an immersed piecewise-real-analytic sphere. The leaf $\ell_{q_{3}}$
of $\tilde{\mathcal{F}}\cap\mathcal{R}$ issued from $q_{3}\in W_{3}$
enters into the space bounded by $\hat{T}$ and, because it is transverse
to the fibers of $\Pi|_{\mathcal{R}}$, must meet the compact $\hat{T}$
at some point. Because $\tilde{\mathcal{F}}\cap\mathcal{R}$ is regular
$\ell_{q_{3}}$ cannot intersect $T$ (invariant by $\tilde{\mathcal{F}}\cap\mathcal{R}$)
and therefore intersects both $W_{2}$ and $W_{3}$. We therefore
obtain a (connected) sub-leaf $\ell_{q_{3}\to q_{2}}\subset\ell_{q_{3}}$
containing both $q_{3}$ and some $q_{2}\in\ell_{q_{3}}\cap\adh{W_{2}}$.
If $q_{3}=p_{3}\in\Gamma_{3}$ we let $\ell_{q_{3}\to q_{2}}:=\ell_{p_{2}\to p_{3}}$
and $q_{2}:=p_{2}$. Define the map
\begin{eqnarray*}
\hol{3\to2}~:~\adh{W_{3}} & \longrightarrow & \adh{W_{2}}\\
q_{3} & \longmapsto & q_{2}
\end{eqnarray*}
whose restriction to $\Gamma_{3}$ is the inverse of $\hol{2\to3}$.
Clearly $\hol{3\to2}$ is injective. We need to show that $\hol{3\to2}$
is continuous, so that it will be an homomorphism onto its image,
which can then only span the whole $\adh{W_{2}}$ for $\hol{3\to2}\left(\Gamma_{3}\right)=\Gamma_{2}$.

\begin{figure}[H]
\subfloat[Accumulation of double intersections (white squares)]{\includegraphics[width=6cm]{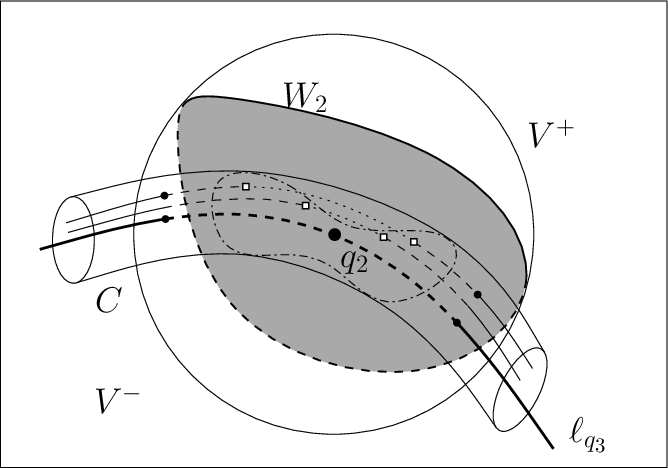}

}\hfill{}\subfloat[Only simple intersections (white discs)]{\includegraphics[width=6cm]{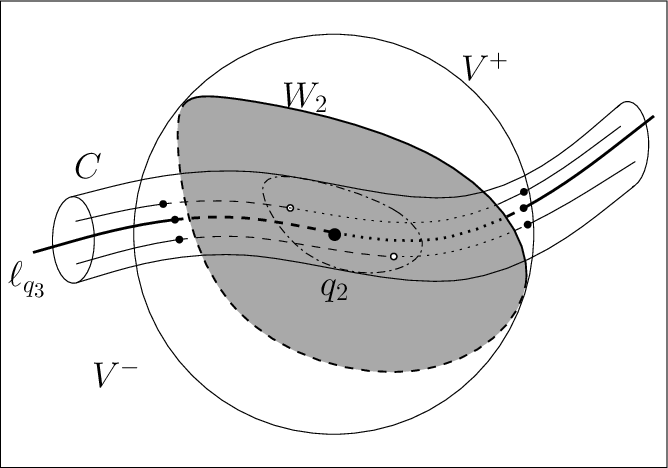}

}

\caption{\label{fig:omega_sconnected_end}}
\end{figure}

First notice that if $q_{2}$ is not a tangency point between $\tilde{\mathcal{F}}\cap\mathcal{R}$
and $\adh{W_{2}}$ then $\hol{3\to2}$ is continuous at $q_{3}$ (proceed
with the argument in a finite covering of $\ell_{q_{3}\to q_{2}}$
by rectifying charts). Now if $\ell_{q_{3}}$ is tangent to $\adh{W_{2}}$
at $q_{2}$ then one among two mutually exclusive situations happens.
In the following $V$ is any sufficiently small neighborhood of $q_{2}$
which is separated by $\adh{W_{2}}$ into two domains $V^{+}$ and
$V^{-}$ with, say, $V^{+}$ outside $\hat{T}$.
\begin{enumerate}
\item The leaf $\ell_{q_{3}}$ does not exit $\hat{T}$ near $q_{2}$, that
is $\ell_{q_{3}}\cap V^{+}=\emptyset$, as in Figure~\ref{fig:omega_sconnected_end}~(A).
\item The leaf pierces $\hat{T}$ at $q_{2}$, \emph{i.e.} $\ell_{q_{3}}\cap V^{\pm}\neq\emptyset$,
as in Figure~\ref{fig:omega_sconnected_end}~(B).
\end{enumerate}
Let $C$ be a small tubular neighborhood of $\ell_{q_{3}}\cap V$
invariant by $\tilde{\mathcal{F}}\cap\mathcal{R}$. Lemma~\ref{lem:W2_transverse}
rules (1) out because in that case there must exist leaves visiting
$V^{+}$ after leaving $V^{-}$ and before entering again in $V^{-}$.
Hence case~(2) is the only possible event and $\hol{3\to2}$ is continuous
as any leaf must exit by $V^{+}$ after arriving through $V^{-}$.\end{proof}
\begin{cor}
For any $z\in W_{0}$ the path $\tilde{\gamma}_{z}$ can be lifted
in $\tilde{\mathcal{F}}$ through $\Pi$ starting from $\left(z,w_{*}\right)$
all the way to $\left(z_{*},\mathcal{D}_{R}\left(z\right)\right)$.
In other words $W_{0}\subset\Omega$ and $\Omega$ is simply connected.\end{cor}
\begin{proof}
Lemma~\ref{lem:surface_W3} asserts that $\tilde{\gamma}_{z}$ can
be lifted at least until $z_{3}$. In the 3-dimensional real slice
$\mathcal{R}':=\left[z_{3},z_{*}\right]\times\ww C$ the saturation
of $\Gamma_{3}$ by the leaves of the regular foliation $\tilde{\mathcal{F}}\cap\mathcal{R}'$
(transverse to the fibers of $\Pi|_{\mathcal{R}'}$) is a smooth cylinder
$T'$ (similar to $T$ in the proof of the previous lemma). Therefore
any leaf of $\tilde{\mathcal{F}}\cap\mathcal{R}'$ starting from $W_{3}$
cannot escape from $T'$ and must reach the transverse line $\left\{ z=z_{3}\right\} $.
\end{proof}

\subsection{\label{sub:boundary}Proof of Proposition~\ref{prop:th_2}~(2)}

A point $z\in\left\{ \re z<\ln\rho\right\} $ belongs to $\Omega$
if, and only if, the lift of the polygonal line $\tilde{\gamma}_{z}$
in $\tilde{\mathcal{F}}$ starting from $\left(z_{0},w_{*}\right)$
does not meet the real $3$-space $\left\{ \re w=\ln r\right\} \subset\ww C^{2}$.
For $z_{0}\in\partial\Omega\backslash\left\{ \re z=\ln\rho\right\} $
we can find a compact neighborhood $W\ni z_{0}$ such that if $r$
were slightly bigger we would have $W\subset\Omega$. In particular
we can assume that the Dulac map is the restriction to $\Omega$ of
an analytic map, written $\mathcal{D}$, on a neighborhood of $W$.
For $z\in W$ we call $\ell_{z}$ the real-analytic curve obtained
as the image of the lift of the path $\tilde{\gamma}_{z}$ in $\tilde{\mathcal{F}}$
starting from $\left(z,w_{*}\right)$ and ending at $\left(z_{*},\mathcal{D}\left(z\right)\right)$.
Then
\begin{eqnarray*}
A_{W} & := & \bigcup_{z\in W}\ell_{z}
\end{eqnarray*}
 is a compact real-analytic $3$-manifold which intersects $\left\{ \re w=\ln r\right\} $
along a compact real surface $\mathcal{S}$ (with boundary). If $z\in W\cap\partial\Omega$
the curve $\ell_{z}$ cannot meet $\left\{ \re w>\ln r\right\} $.
Therefore $\ell_{z}$ intersects $\mathcal{S}$ in a finite numbers
of points $\left(p_{n}^{z}\right)_{1\leq n\leq d}$ with $d=d\left(z\right)\geq1$,
all of them tangency points. According to~(\ref{eq:variation_w}),
assuming $\Pi\left(p_{n}^{z}\right)$ lies on the line segment $\left[\left(z\right)_{j},\left(z\right)_{j+1}\right]$
of direction $\theta_{j}\in\ww S^{1}$, the tangency point $p_{n}^{z}$
lies in 
\begin{eqnarray*}
\mathcal{T}_{j} & := & \left\{ p\in\ww C^{2}~:~\re{\frac{\theta_{j}}{\lambda}\left(1+R\circ\mathcal{E}\left(p\right)\right)}=0\right\} \cap\mathcal{S}~.
\end{eqnarray*}
If $\mathcal{T}_{j}=\mathcal{S}$ then $R=0$ and the result is clear.
Otherwise $\mathcal{T}:=\bigcup_{j}\mathcal{T}_{j}$ is a finite union
of irreducible real-analytic curves, as is $\left\{ z\in W~:~\ell_{z}\cap\mathcal{S}\subset\mathcal{T}\right\} $.
The latter contains $\partial\Omega\cap W$. 
\begin{rem}
In case $R$ is merely $C^{1}$ we cannot guarantee that $\ell_{z}\cap\mathcal{S}$
is finite.
\end{rem}

\subsection{Proof of Proposition~\ref{prop:th_2}~(3)}

From Lemma~\ref{lem:dual_searchlight} follows the fact that $\adh{\Omega}\cap\left\{ \re z=\ln\rho\right\} $
is a nonempty line segment, as can be seen by adapting in a straightforward
way the proof of Lemma~\ref{lem:boundary_z_FLF}. In particular $\Omega$
is connected. To prove that $\adh{\Omega}\cap\left\{ \re z=\ln\rho\right\} $
can be arbitrarily wide provided that $\re{w_{*}}$ be sufficiently
small it is sufficient to invoke the fact that $\left\{ y=0\right\} $
is the adherence of a separatrix of $\fol[R]$, so that $\grp$ contains
elements winding more and more around $\left\{ x=0\right\} $.

\subsection{Proof of Proposition~\ref{prop:th_2}~(4)}

Because $S\left(z_{*},-\vartheta,\delta\right)\subset\Omega_{*}$
we only need to ensure that $\im{\vartheta_{+}}$ and $\im{\vartheta_{-}}$
have opposite signs, where $\vartheta_{\pm}$ is defined by~(\ref{eq:theta_mp}).
This can be enforced by taking $\re{z_{*}}$ and $\re{w_{*}}$ negative
enough, \emph{i.e.} by taking $\delta$ as close to $\frac{\pi}{2}$
as needed.

\section{\label{sec:asymptotics}Asymptotics of the Dulac map}
\begin{defn}
A \textbf{time-form} of a vector field $X$ is a meromorphic $1$-form
$\tau$ such that $\tau\left(X\right)=1$.
\end{defn}
Because of the specific form of $X_{R}$ one can always choose a time-form
as 
\begin{eqnarray*}
\tau & := & \frac{\dd x}{\lambda x}\,.
\end{eqnarray*}
We prove first in Section~\ref{sub:Integral-representation} that
to obtain the image of $x$ by the Dulac map $\mathcal{D}_{R}$ we
need to compute the integral 
\begin{eqnarray*}
\int_{\gamma_{R}\left(x\right)}R\tau &  & \,,
\end{eqnarray*}
where $\gamma_{R}$ is a path tangent to $\fol$ linking $\left(x,y_{*}\right)$
to some point of $\Pi^{-1}\left(x_{*}\right)$. This is~(1) of the
Proximity Theorem. We intend in a second step (Section~\ref{sub:Approximation})
to compare this value with that of
\begin{eqnarray*}
\int_{\gamma_{0}\left(x\right)}R\tau
\end{eqnarray*}
which can be explicitly computed (Section~\ref{sub:Computations}).
We prove more generally the quantitative result:
\begin{thm}
\label{thm:asymptotics}Let $N\in\ww N_{>0}$ be given. There exists
a constant $M>0$ depending only on $N,\,\lambda,$ $a$, $\rho$,
$\norm R$ and $\delta$ such that for any $G\in\mathcal{O}\left(\mathcal{U}\right)\cap x^{a}\germ{x,y}$
with $\frac{\partial G}{\partial y}$ bounded and all $x=\exp z$
with $\left|\im{z-z_{*}}\right|\leq\pi N$ one has 
\begin{eqnarray*}
\left|\int_{\gamma_{R}\left(x\right)}G\tau-\int_{\gamma_{0}\left(x\right)}G\tau\right| & \leq & M\norm{\frac{\partial G}{\partial y}}\left|y_{*}\right|\left|x\right|^{a}\,.
\end{eqnarray*}
\end{thm}
\begin{rem}
Under the assumption $\re{a+\frac{1}{\lambda}}>0$ we have $\left|x\right|^{a}=o\left(\left|x^{-\nf 1{\lambda}}\log x\right|\right)$,
proving~(2) of the Proximity Theorem when $G:=R$.
\end{rem}

\subsection{\label{sub:Integral-representation}Integral expression of the Dulac
map: proof of Proximity Theorem~(1)}

\begin{lem}
\label{lem:X_primitive} Let $\Sigma$ be a transverse to $\fol[R]$;~we
refer to Definition~\ref{def_holonomy} for the construction of the
groupoid $\grp$. For given $G\in\mathcal{O}\left(\mathcal{U}\right)$
the integration process 
\begin{eqnarray*}
F\,:\,\gamma\in\grp & \longmapsto & \int_{\gamma}G\tau
\end{eqnarray*}
gives rise to a holomorphic function whose Lie derivative $X_{R}\cdot F$
along $X_{R}$ can be computed by considering $F$ as a local analytic
function of the end-point $\ept$. Then 
\begin{eqnarray*}
X_{R}\cdot F & = & G\,.
\end{eqnarray*}
\end{lem}
\begin{proof}
Outside the singular locus of $X_{R}$ there exists a local rectifying
system of coordinates: a one-to-one map $\psi$ such that $\psi^{*}X_{R}=\pp t$.
In these coordinates we have $\psi^{*}\left(G\tau\right)=G\circ\psi\dd t$.
The fundamental theorem of integral calculus yields the result.
\end{proof}
Notice that $X_{0}$ admits a (multivalued) first-integral with connected
fibers
\begin{eqnarray*}
H_{0}\left(x,y\right) & := & x^{-\nf 1{\lambda}}y\,,
\end{eqnarray*}
which means that it lifts through $\sigma$ to a holomorphic map,
still written $H_{0}$, constant along the leaves of $\sigma^{*}\fol[0]$
and whose range is in one-to-one correspondence with the space of
leaves of $\sigma^{*}\fol[0]$.
\begin{lem}
\label{lem:first-int}Let $\Sigma$ be a transverse to $\fol$. The
function
\begin{eqnarray*}
H_{R}\,:\,\grp & \longrightarrow & \ww C\\
\gamma & \longmapsto & H_{0}\exp\int_{\gamma}-R\tau
\end{eqnarray*}
 is a holomorphic first-integral of $\sigma^{*}\fol$ with connected
fibers.\end{lem}
\begin{proof}
The fact that $H_{R}$ is holomorphic on $\grp$ is clear enough from
Lemma~\ref{lem:X_primitive}. It is a first integral of $\sigma^{*}\fol$
if, and only if, the Lie derivative $X_{R}\cdot H_{R}$ vanishes.
This quantity is computed as follows:
\begin{eqnarray*}
X_{R}\cdot H_{R} & = & X_{0}\cdot H_{R}+R\left(y\pp y\right)\cdot H_{R}\\
 & = & H_{R}\times\left(X_{R}\cdot\int_{\gamma}-R\tau+R\frac{\left(y\pp y\right)\cdot H_{0}}{H_{0}}\right)\,.
\end{eqnarray*}
Since $\left(y\pp y\right)\cdot H_{0}=H_{0}$ our claim holds. The
fact that $H_{R}$ has connected fibers is a direct consequence of
both facts that $H_{0}$ also has and $H_{R}|_{\Sigma}=H_{0}$.\end{proof}
\begin{cor}
\label{cor:th_1}We have
\begin{eqnarray*}
\mathcal{D}_{R} & = & \mathcal{D}_{0}\times\exp\int_{\bullet}R\tau\,.
\end{eqnarray*}
\end{cor}
\begin{proof}
For any path $\gamma\in\Gamma^{*}$ we have the relation $H_{R}\left(\spt\right)=H_{R}\left(\gamma\right)$,
that is
\begin{eqnarray*}
H_{0}\left(\gamma\right)\exp\int_{\gamma}-R\tau & = & H_{0}\left(\spt[\gamma]\right)\,.
\end{eqnarray*}
The conclusion follows since $\gamma\mapsto H_{0}\left(\gamma\right)$
is linear with respect to the $y$-coordinate of $\gamma$ when $x_{*}$
is fixed.
\end{proof}

\subsection{\label{sub:Approximation}Approximation to the formal model}

We fix once and for all a preimage $\left(z_{*},w_{*}\right)\in\mathcal{E}^{-1}\left(x_{*},y_{*}\right)$.
Since $\mathcal{D}_{R}$ is naturally defined on the universal covering
$\tilde{\mathcal{U}}$ of $\hat{\mathcal{U}}:=\mathcal{U}\backslash\left\{ xy=0\right\} $
we keep on working in logarithmic coordinates
\begin{eqnarray*}
\left(x,y\right) & = & \mathcal{E}\left(z,w\right)=\left(\exp z,\exp w\right)\,.
\end{eqnarray*}
 Notice that the time form $\tau$ is transformed into 
\begin{eqnarray*}
\tilde{\tau} & :=\mathcal{E}^{*}\tau= & \frac{1}{\lambda}\dd z\,.
\end{eqnarray*}

We make here the hypothesis that $\re{\lambda}<0$. We need to compare
this integral and the one obtained for the model, \emph{i.e.} bound
\begin{eqnarray*}
\Delta\left(z_{0}\right) & := & \int_{\tilde{\gamma}\left(z_{0}\right)}\tilde{G}\left(z,w_{R}\left(z,z_{0}\right)\right)-\tilde{G}\left(z,w_{0}\left(z,z_{0}\right)\right)\dd z
\end{eqnarray*}
where $\tilde{\gamma}\left(z_{0}\right)$ is a path linking $z_{0}$
to $z_{*}$ within $\Omega_{\left(z_{0},w_{*}\right)}$ and $z\mapsto w_{R}\left(z,z_{0}\right)$
is its lift in $\tilde{\fol[R]}$ starting from $\left(z_{0},w_{*}\right)$.
We mention that
\begin{eqnarray*}
w_{0}\left(z,z_{0}\right) & = & w_{*}+\frac{z-z_{0}}{\lambda}\,.
\end{eqnarray*}
For any $\left(z,w_{j}\right)\in\tilde{\mathcal{U}}$ we have the
estimate 
\begin{eqnarray*}
\left|\tilde{G}\left(z,w_{2}\right)-\tilde{G}\left(z,w_{1}\right)\right| & \leq & \left|\exp\left(az\right)\right|\norm{\frac{\partial G}{\partial y}}\left|\exp w_{2}-\exp w_{1}\right|
\end{eqnarray*}
 so that 
\begin{eqnarray*}
\left|\Delta\left(z_{0}\right)\right| & \leq & \norm{\frac{\partial G}{\partial y}}\int_{\tilde{\gamma}}\left|\exp\left(az+w_{0}\left(z,z_{0}\right)\right)\left(\exp\left(w_{R}\left(z,z_{0}\right)-w_{0}\left(z,z_{0}\right)\right)-1\right)\dd z\right|\,.
\end{eqnarray*}
Setting
\begin{eqnarray*}
D_{R}\left(z,z_{0}\right) & := & \left|w_{R}\left(z,z_{0}\right)-w_{0}\left(z,z_{0}\right)\right|
\end{eqnarray*}
and taking $\left|\exp z-1\right|\leq\left|z\right|\exp\left|z\right|$
into account we derive 
\begin{eqnarray*}
\left|\Delta\left(z_{0}\right)\right| & \leq & \norm{\frac{\partial G}{\partial y}}\int_{\tilde{\gamma}}\exp\re{az+w_{0}\left(z,z_{0}\right)}D_{R}\left(z,z_{0}\right)\exp D_{R}\left(z,z_{0}\right)\left|\dd z\right|\,.
\end{eqnarray*}
The proof is done when the next lemma is established:
\begin{lem}
There exists a constant $K>0$, depending only on $N,\,\lambda,$
$a$, $\rho$, $\norm R$, $z_{*}$ and $\delta$, such that 
\begin{eqnarray*}
\sup_{t}D_{R}\left(\tilde{\gamma}\left(t\right),z_{0}\right) & \leq & K
\end{eqnarray*}
 where $\tilde{\gamma}$ is the integration path built in Proposition~\ref{prop:integ_path}.
The values of $K$ is explicitly, if crudely, determined in the proof
to come.\end{lem}
\begin{proof}
Invoking the estimate~(\ref{eq:growth_estim_FLF}) from Lemma~\ref{lem:_growth_estim_FLF}
and setting 
\begin{eqnarray*}
C_{1} & := & \frac{\norm R}{a\left|\lambda\right|}\\
C_{2} & := & \frac{C_{1}}{\re{\vartheta_{+}}}\\
C_{3} & := & a\rho^{a}C_{1}
\end{eqnarray*}
we know that, using the number $\kappa$ obtained in Proposition~\ref{prop:integ_path},
\begin{eqnarray*}
\sup_{z\in\left[z_{0},z_{1}\right]}D_{R}\left(z,z_{0}\right) & \leq & K_{1}:=C_{1}\left(\rho^{a}+\exp\left(a\kappa\right)\right)\\
\sup_{z\in\left[z_{1},z_{2}\right]}D_{R}\left(z,z_{0}\right) & \leq & K_{2}:=K_{1}+C_{2}\left(\rho^{a}+\exp\left(a\kappa\right)\right)\\
\sup_{z\in\left[z_{2},z_{3}\right]}D_{R}\left(z,z_{0}\right) & \leq & K_{3}:=K_{2}+C_{3}\left(2\pi N+\tan\left(\left|\arg\vartheta_{+}\right|\right)\left(\ln\rho-\kappa\right)\right)\\
\sup_{z\in\left[z_{3},z_{*}\right]}D_{R}\left(z,z_{0}\right) & \leq & K:=K_{3}+C_{2}\left(\rho^{a}+\exp\re{az_{*}}\right)\,.
\end{eqnarray*}

\end{proof}
We conclude now the proof starting from 
\begin{eqnarray*}
\left|\Delta\left(z_{0}\right)\right| & \leq & K\norm{\frac{\partial G}{\partial y}}\exp\re{K+w_{*}-\nf{z_{0}}{\lambda}}\int_{\tilde{\gamma}}\exp\re{\left(a+\frac{1}{\lambda}\right)z}\left|\dd z\right|\,.
\end{eqnarray*}
Let $\alpha:=\left|a+\frac{1}{\lambda}\right|$. We bound each partial
integral $I_{\star\to\bullet}:=\int_{\left[z_{\star},z_{\bullet}\right]}\exp\re{\left(a+\frac{1}{\lambda}\right)z}\left|\dd z\right|$
in the following manner:
\begin{eqnarray*}
I_{0\to1} & \leq & \exp\re{\left(a+\frac{1}{\lambda}\right)z_{0}}\int_{0}^{\kappa-\re{z_{0}}}\exp\left(t\re{a+\frac{1}{\lambda}}\right)\dd t\\
 & \leq & \frac{\exp\re{\left(a+\frac{1}{\lambda}\right)\kappa}-\exp\re{\left(a+\frac{1}{\lambda}\right)z_{0}}}{\re{a+\frac{1}{\lambda}}}\,,\\
I_{1\to2} & \leq & \exp\re{\left(a+\frac{1}{\lambda}\right)z_{1}}\int_{0}^{\ln\rho-\kappa}\exp\left(t\re{\left(a+\frac{1}{\lambda}\right)\vartheta_{+}}\right)\dd t\\
 & \leq & \exp\left(\alpha\left|z_{1}\right|\right)\frac{\exp\left(\alpha\left(\ln\rho-\kappa\right)\right)-1}{\alpha}\\
 & \leq & \exp\left(\alpha\left(\sqrt{\kappa^{2}+\pi^{2}N^{2}}+\ln\rho-\kappa\right)\right)\,,\\
I_{2\to3} & \leq & \exp\left(\alpha\left|z_{2}\right|\right)\int_{0}^{\im{z_{3}-z_{2}}}\exp\left(t\left|\im{\frac{1}{\lambda}}\right|\right)\left|\dd t\right|\\
 & \leq & \exp\left(\alpha\left|z_{2}\right|+\left|\frac{1}{\lambda}\right|\left|\im{z_{3}-z_{2}}\right|\right)\\
 & \leq & \exp\left(\alpha\left(\left|z_{*}\right|+\pi N+\tan\left(\left|\arg\vartheta\right|+\delta\right)\left(\ln\rho-\kappa\right)\right)\right)\\
 &  & \,\,\,\,\,\times\exp\left(\left|\frac{1}{\lambda}\right|\left(2\pi N+\tan\left(\left|\arg\vartheta\right|+\delta\right)\left(\ln\rho-\kappa\right)\right)\right)\,,\\
I_{3\to*} & \leq & \exp\left(\alpha\left(\sqrt{\re{z_{*}}^{2}+\pi^{2}N^{2}}+\ln\rho-\re{z_{*}}\right)\right)\,.
\end{eqnarray*}
In particular the dominant integral in the above list is $I_{0\to1}$,
so that there exists a constant $M$, satisfying the required dependency
properties, with
\begin{eqnarray*}
\left|\Delta\left(z_{0}\right)\right| & \leq & M\norm{\frac{\partial G}{\partial y}}\exp\re{w_{*}+az_{0}}\,.
\end{eqnarray*}
Since $\re{a+\frac{1}{\lambda}}>0$ we have
\begin{eqnarray*}
\Delta\left(z_{0}\right) & = & o\left(\left|z_{0}\right|\exp\re{\frac{-z_{0}}{\lambda}}\right)
\end{eqnarray*}
as expected.

\subsection{\label{sub:Computations}Study of the model}

\subsubsection{Explicit computation}

We want to compute for $n,\,m\,\in\ww N$ the functions defined by
\begin{eqnarray*}
T_{n,m}\left(z\right) & := & \int_{\tilde{\gamma}\left(z\right)}\exp\left(nu+mw_{0}\left(u,z_{0}\right)\right)\dd u\\
 & = & \exp\left(m\left(w_{*}-\frac{z}{\lambda}\right)\right)\times\int_{z}^{z_{*}}\exp\left(\left(n+\frac{m}{\lambda}\right)u\right)\dd u\,.
\end{eqnarray*}
If $n+\nf m{\lambda}=0$ then $\lambda=-\nf pq$, with $p$ and $q$
co-prime positive integers, and $\left(n,m\right)=k\left(q,p\right)$
with $k\in\ww N$. In that case, and when $k>0$,
\begin{eqnarray}
T_{kq,kp}\left(z\right) & = & \left(z_{*}-z\right)\exp\left(k\left(pw_{*}+qz\right)\right)=O\left(\left|z\exp\re{az}\right|\right)\,.\label{eq:rational_estim}
\end{eqnarray}
The other case $n+\nf m{\lambda}\neq0$ is not more difficult:
\begin{eqnarray*}
T_{n,m}\left(z\right) & = & \exp\left(mw_{*}+nz\right)\frac{\exp\left(\left(n+\nf m{\lambda}\right)\left(z_{*}-z\right)\right)-1}{n+\nf m{\lambda}}\,.
\end{eqnarray*}
One can see easily that as $n+\nf m{\lambda}$ tends to zero (which
may happen if, and only if, $\lambda$ is a negative irrational) the
function $T_{n,m}$ grows in modulus. The dominant support introduced
in Definition~\ref{def_dominant_support} allows to discriminate
between two kinds of growth rate.

\subsubsection{Resonant support}

We show now that the resonant support consists of (quasi-)resonant
monomials only.
\begin{lem}
\label{lem:dominant_support}Assume that $\lambda<0$ and $a+\frac{1}{\lambda}>0$. 
\begin{enumerate}
\item If $\lambda=-\nf pq<0$ is a rational number then 
\begin{eqnarray*}
\tt{Res}\left(a,\lambda\right) & = & \left\{ k\left(q,p\right)\,:\,k\in\ww N\,,\,kq\geq a\right\} \,.
\end{eqnarray*}

\item If $\lambda$ is a negative irrational we denote by $\left(-\nf{p_{k}}{q_{k}}\right)_{k\in\ww N}$
its sequence of convergents. Then 
\begin{eqnarray*}
\tt{Res}\left(a,\lambda\right) & = & \left\{ \left(q_{k},p_{k}\right)\,:\,k\in\ww N\,,\,q_{k}\geq a\right\} \,.
\end{eqnarray*}

\end{enumerate}
\end{lem}
\begin{proof}
~
\begin{enumerate}
\item Because we have $n\geq a\geq q$ the relation $\left|n\lambda+m\right|<\frac{1}{2n}$
becomes 
\begin{eqnarray*}
\left|np-mq\right| & < & \frac{q}{2n}<1\,.
\end{eqnarray*}
Hence $np=mq$ and since $p$ and $q$ are co-prime the conclusion
follows.
\item This is a consequence of the well-known result in continued-fraction
theory: if $\frac{p}{q}\in\ww Q_{>0}$ is given such that $\left|\frac{p}{q}+\lambda\right|<\nf{q^{-2}}2$
then $\left(p,q\right)$ is one of the convergents of $\left|\lambda\right|$.
\end{enumerate}
\end{proof}

\subsubsection{\label{sub:prop}Dominant terms: proof of Proposition~\ref{prop:model_dominant}}

Nothing needs to be proved for $G_{0}$ so we assume that $G$ expands
into a power series $G\left(x,y\right)=\sum_{n\geq a,m>0}G_{n,m}x^{n}y^{m}$
convergent on a closed polydisc of poly-radii at least $\left(\rho+\epsilon,r+\epsilon\right)$.
Because of the Cauchy formula, for all $n,\,m$
\begin{eqnarray*}
\left|G_{n,m}\right| & \leq & C\left(\rho+\epsilon\right)^{-n}\left(r+\epsilon\right)^{-m}
\end{eqnarray*}
where $C:=\sup_{\left|x\right|=\rho+\epsilon\,,\,\left|y\right|=r+\epsilon}\left|G\left(x,y\right)\right|$. 

If $\lambda$ is not real then 
\begin{eqnarray*}
\inf_{\left(n,m\right)\in\ww N^{2}\backslash\left\{ \left(0,0\right)\right\} }\left|n\lambda+m\right| & \geq & a\left|\im{\lambda}\right|>0\,.
\end{eqnarray*}
Let $z-z_{*}$ be given with imaginary part bounded by $N\pi$ for
some integer $N>0$ and with real part lesser than 
\begin{eqnarray*}
\mu & := & -\left|\frac{\im{\lambda}}{\re{\lambda}}\right|N\pi\,.
\end{eqnarray*}
 By construction of $\mu$ we have
\begin{eqnarray*}
\re{\frac{m}{\lambda}\left(z_{*}-z\right)} & = & \frac{m}{\left|\lambda\right|^{2}}\left(\re{\lambda}\re{z_{*}-z}+\im{\lambda}\im{z_{*}-z}\right)<0\\
 & \leq & \re{\frac{1}{\lambda}\left(z_{*}-z\right)}
\end{eqnarray*}
so that we derive at last 
\begin{eqnarray*}
\left|T_{n,m}\left(z\right)\right| & \leq & \frac{2\left|\lambda\right|r^{m}\rho^{n+\re{\nf 1{\lambda}}}}{a\left|\im{\lambda}\right|}\exp\re{-\frac{z}{\lambda}}\,.
\end{eqnarray*}
Now
\begin{eqnarray*}
\left|\int_{z}^{z_{*}}G\circ\mathcal{E}\dd z\right| & = & \left|\sum_{n\geq a\,,\,m>0}G_{n,m}T_{n,m}\left(z\right)\right|\\
 & \leq & \frac{2\left|\lambda\right|\rho^{\re{\nf 1{\lambda}}}C}{a\left|\im{\lambda}\right|}\times\frac{\left(r+\epsilon\right)\left(\rho+\epsilon\right)}{\epsilon^{2}}\times\exp\re{-\frac{z}{\lambda}}\\
 & = & O\left(\left|\exp\nf{-z}{\lambda}\right|\right)\,,
\end{eqnarray*}
ending the proof for the non-real case. In fact this reasoning goes
on holding even when $\lambda<0$ as long as $\left(n,m\right)$ belongs
not to $\rsupp$, since in that case 
\begin{eqnarray*}
\left|n+\nf m{\lambda}\right| & > & \frac{\left|\lambda\right|}{2n}
\end{eqnarray*}
 has strictly sub-geometric inverse. 

Take now $\lambda$ negative real and $G=G_{\tt{Res}}$. If $\lambda=-\nf pq$
is a negative rational number then~(\ref{eq:rational_estim}) provides
what remains to be proved. Assume now that $\lambda$ is irrational.
Because $\left|\exp z-1\right|\leq\left|z\right|$ when $\re z<0$
we have for $\left(n,m\right)\in\rsupp$ 
\begin{eqnarray*}
\left|T_{n,m}\left(z\right)\right| & \leq & \left|\left(z-z_{*}\right)\exp\left(nz\right)\right|r^{m}\,.
\end{eqnarray*}
Therefore 
\begin{eqnarray*}
\left|\int_{\gamma_{0}\left(z\right)}G_{\tt{Res}}\circ\mathcal{E}\dd z\right| & \leq & C\sum_{\left(n,m\right)\in\tt{Res}\left(a,\lambda\right)}\left(\rho+\epsilon\right)^{-n}\left(r+\epsilon\right)^{-m}\left|T_{n,m}\left(z\right)\right|\\
 & \leq & C\left|z_{*}-z\right|\sum_{\left(n,m\right)\in\tt{Res}\left(a,\lambda\right)}\left(\frac{\exp\re z}{\rho+\epsilon}\right)^{n}\left(\frac{r}{r+\epsilon}\right)^{m}\\
 & \leq & C\frac{\left(r+\epsilon\right)\left(\rho+\epsilon\right)}{\epsilon^{2}}\left|z_{*}-z\right|\exp\re{az}\\
 & = & O\left(\left|z\exp\nf{-z}{\lambda}\right|\right)
\end{eqnarray*}
as expected.

\bibliographystyle{preprints}
\bibliography{bibliography}

\end{document}